\documentclass[10pt]{amsart}
\usepackage{amsmath,amssymb,xypic,array}
\usepackage[T1]{fontenc}
\usepackage{amsfonts}
\usepackage{amsthm}
\usepackage{setspace}
\usepackage{tikz}
\usepackage{booktabs}
\usepackage{longtable}
\usepackage{geometry}
\usetikzlibrary{trees}

\usepackage[english]{babel}
\usepackage[latin1]{inputenc}
\usepackage{amsmath,amssymb,amsthm,epsfig}
\usepackage{color}
\usepackage{mathrsfs}
\usepackage{latexsym}
\usepackage{hyperref}
\usepackage{textcomp}

\newtheorem{Theorem}{Theorem}[section]
\newtheorem{Lemma}[Theorem]{Lemma}
\newtheorem{Prop}[Theorem]{Proposition}

\newtheorem{Cor}[Theorem]{Corollary}
\newtheorem{Ex}[Theorem]{Example}

\newtheorem{Obs}[Theorem]{Remark}

\newtheorem*{conjecture}{{\bf Conjecture}}

\newtheorem*{thm-intro}{Theorem}
\newtheorem{appxthm}{Theorem}[section]
\newtheorem{appxlem}[appxthm]{Lemma}
\newtheorem{appxcor}[appxthm]{Corollary}

\def\C{{\mathbb{C}}}

\def\P{{\mathbb{P}}}
\def\Q{{\mathbb{Q}}}

\def\Z{{\mathbb{Z}}}

\def\sF{{\mathscr{F}}}
\def\sG{{\mathscr{G}}}

\def\G{{\mathscr{G}}}
\def\D{{\mathcal{D}}}

\def\O{{\mathcal{O}_{\mathbb{P}^3}}}
\def\OZ{{\mathcal{O}_{Z}}}
\def\IZ{I_{Z/{\mathbb{P}^3}}}
\def\OC{{\mathcal{O}_{C}}}
\def\IC{I_{C/{\mathbb{P}^3}}}
\def\OZX{{\mathcal{O}_{Z}}}
\def\OCX{{\mathcal{O}_{C}}}
\def\OX{{\mathcal{O}_{X}}}

\def\a{{\alpha}}
\def\b{{\beta}}

\def\o{\omega}

\def\Om{{\Omega_{\mathbb{P}^3}^1}}

\newcommand{\op}[1]{{\mathcal O}_{\mathbb{P}^{#1}}}
\newcommand{\p}[1]{{\mathbb{P}^{#1}}}

\newcommand{\quotf}{{\mathcal Q}{\it uot}}
\newcommand{\dist}{{\mathcal D}{\it ist}}
\newcommand{\sch}{\mathfrak{Sch}_{/\C}}
\newcommand{\sets}{\mathfrak{Sets}}

\newcommand{\inext}{{\mathcal E}{\it xt}}
\DeclareMathOperator{\sing}{Sing}
\DeclareMathOperator{\supp}{Supp}
\DeclareMathOperator{\Hom}{Hom}
\DeclareMathOperator{\Ext}{Ext}
\DeclareMathOperator{\rk}{{rk}}
\DeclareMathOperator{\coker}{{coker}}
\DeclareMathOperator{\img}{{im}}
\DeclareMathOperator{\codim}{{codim}}
\DeclareMathOperator{\Pic}{{Pic}}
\def\BB{\textrm{BB}}

\title[Codimension one holomorphic distributions on the projective three-space]
{Codimension one holomorphic distributions on the projective three-space}
\author[ O. Calvo-Andrade ]{ O. Calvo-Andrade }
\thanks{ }
\dedicatory{}
\address{CIMAT \\ Ap. Postal 402, Guanajuato, 36000, Gto. Mexico}
\email{omegar.mat@gmail.com}

\author[M. Corr\^ea]{M. Corr\^ea}
\thanks{ }
\dedicatory{}
\address{ICEX-UFMG \\ Departamento de Matem\'atica \\
Av. Ant\^onio Carlos, 6627 \\ 31270-901 Belo Horizonte-MG, Brazil}
\email{mauriciomatufmg@gmail.com}
\author[  M. Jardim]{M. Jardim}
\thanks{ }
\dedicatory{}
\address{IMECC - UNICAMP \\ Departamento de Matem\'atica \\
Rua S\'ergio  Buarque de Holanda, 651\\ 13083-970 Campinas-SP, Brazil}
\email{jardim@ime.unicamp.br}

\keywords{}
\subjclass{}
\date{}

\begin{document}

\begin{abstract}
We study codimension one holomorphic distributions on the projective three-space, analyzing the properties of their singular schemes and tangent sheaves. In particular, we provide a classification of codimension one distributions of degree at most 2 with locally free tangent sheaves, and show that codimension one distributions of arbitrary degree with only isolated singularities have stable tangent sheaves. Furthermore, we describe the moduli space of distributions in terms of Grothendieck's Quot-scheme for the tangent bundle. In certain cases, we show that the moduli space of codimension one distributions on the projective space is an irreducible, nonsingular quasi-projective variety. Finally, we prove that every rational foliation, and certain logarithmic foliations have stable tangent sheaves.
\end{abstract}

\maketitle

\section{Introduction}
Techniques from algebraic geometry have been extremely useful in the study of singular holomorphic foliations, see for instance \cite{AD,BMc,CJV,EK,GM,J,LPT,Mi}. 
In particular, Jouanolou classified codimension one foliations on $\P^3$ of degrees 0 and 1 in his monograph \cite{J}; Cerveau and Lins Neto showed in \cite{CL} that there exist six irreducible components of foliations of degree 2 on projective spaces; and Polishchuk, motivated by the study of holomorphic Poisson structures, also found in \cite{P} a classification of foliations of degree 2 on $\P^3$ under certain hypotheses on the singular set of the foliations.

In this work, we study holomorphic distributions on $\P^3$. We focus on studying properties of their singular schemes, on analyzing the semistability of the tangent sheaf, and on describing certain moduli spaces of distributions of low degree.

More precisely, we provide a complete classification of codimension one distributions of degrees 0 (for the sake of completeness) and 1 in Sections \ref{class0} and \ref{class1}, respectively. We describe all possible singular schemes, and prove that the tangent sheaf of a codimension one distribution of degree 1 is stable whenever it does not split as a sum of line bundles; in particular, we show that if the tangent sheaf of a codimension one distribution of degree 1 is locally free, then it splits as a sum of line bundles. Finally, we describe one moduli space of codimension one distributions of degree 0 in Section \ref{sec:nc-dist}, and one moduli space of codimension one distributions of degree 1 in Section \ref{sec:dist deg 1}; both are irreducible, nonsingular quasi-projective varieties, of dimensions 5 and 14, respectively.

Regarding codimension one distributions of degree 2, we list all possible values for the second and third Chern classes of the tangent sheaves, see Table \ref{deg 2 table}. In particular, we conclude that the tangent sheaf of such distributions is always $\mu$-semistable.

In addition, we classify codimension one distributions of degree 2 whose tangent sheaf is locally free. In the statement below, we say that a curve of degree $m$ and genus $p$ is \emph{reduced up to deformation} if it lies in a component of the Hilbert scheme ${\rm Hilb}_{m,p}$ whose general point corresponds to a reduced curve. 

\begin{thm-intro} %\label{class2intro}
Let $\sF$ be a codimension one distribution of degree $2$ on $\p3$ with locally free tangent sheaf $T_\sF$, and such that $\sing(\sF)$  is reduced, up to deformation. Then:
\begin{enumerate}
\item  $T_\sF$ splits as a sum of line bundles and 
\begin{enumerate}

\item either $T_\sF=\O(1)\oplus \O(-1)$, and $\sing(\sF)$ is a connected curve of degree $7$ and arithmetic genus $5$.

\item or $T_\sF=\O\oplus \O$, and $\sing(\sF)$ is a connected curve of degree $6$ and arithmetic genus $3$.

\end{enumerate}

\item  $T_\sF$ is stable, and:

\begin{enumerate}
\item either $T_\sF$ is a null-correlation bundle, and $\sing(\sF)$ is a connected curve of degree $5$ and arithmetic genus $1$; in addition, $\sF$ possesses sub-foliations of degree $2$ which are singular along two skew lines;
 
\item or $T_\sF$ is an instanton bundle of charge $2$, and $\sing(\sF)$ is the disjoint union of a line and a twisted cubic (or a degeneration of such curve); in addition, $\sF$ possesses  sub-foliations of degree $2$ which are singular along three skew lines.  
\end{enumerate}
\end{enumerate}

In addition, this result is effective, in the sense that there exist injective morphisms $\phi : T_\sF \to T\P^3$ with torsion free cokernel for each of the possibilities listed above.  
\end{thm-intro}

In particular, note that if the tangent sheaf of a codimension one distribution of degree 2 is locally free, then either it splits as a sum of line bundles, or it is stable. Moduli spaces of codimension one distribution of degree 2 are described in Sections \ref{sec:nc-dist},  \ref{sec:2-inst},  \ref{sec:r(2,2)}, \ref{sec:r(1,3)}. In all cases considered, the moduli spaces are irreducible, nonsingular quasi-projective varieties. 

Another important open problem in the theory of holomorphic foliations is a conjecture due to Cerveau \cite{Ce}: if $\sF$ is a codimension one foliation on $\P^3$, then the pure 1-dimensional component of $\sing(\sF)$ is connected. This statement is false for non-integrable distributions, since, according to the previous Theorem, there exists a codimension one distribution of degree 2 with locally free tangent sheaf whose singular scheme is not connected. We prove the following cohomological criterion for the connectedness of the singular scheme of codimension one distributions.

\begin{thm-intro} 
Let $\sF$ be a codimension one distribution on $\p3$ of degree $d>0$ with singular scheme $Z$. If $h^2(T_\sF(-2-d))=0$, then $Z$ is connected and of pure dimension 1, so that $T_\sF$ is locally free. Conversely, for $d\ne2$, if $Z$ is connected, then $T_\sF$ is locally free and $h^2(T_\sF(-2-d))=0$.
\end{thm-intro}

In addition, we also explore the relation between connectedness and integrability.

\begin{thm-intro}
Let $\sF$ be a codimension one distribution on $\p3$ of degree $d>0$ with locally free tangent sheaf. If $\sing(\sF)$ is smooth and connected, then
 $\sF$ is not integrable. 
\end{thm-intro}

A key fact in our description of moduli spaces of codimension one distributions of low degree is precisely the stability of the tangent sheaf. The $\mu$-semistability of the tangent sheaf of foliations has also been explored in \cite{AD,LPT} to show the existence of algebraic, rationally connected leaves via the Miyaoka--Bogomolov--McQuillan Theorem \cite{BMc,Mi}. We prove that the tangent sheaf of a codimension one distributions with only isolated singularities is always stable.

\begin{thm-intro}
Let $\sF$ be a codimension one distribution on $\p3$. If $\sing(\sF)$ is a nonempty, 0-dimensional scheme, then $T_\sF$ is stable.
\end{thm-intro}

In addition, we prove that every rational foliation has stable tangent sheaf, see Theorem \ref{racionais} below, and that the tangent sheaf of a logarithmic foliation of degree at most 2 either splits as a sum of line bundles, or it is $\mu$-semistable, see Theorem \ref{logStable}. Such observation motivates the following conjecture.

\begin{conjecture}
If the tangent sheaf of a codimension one foliation on $\p3$ is not split, then it is $\mu$-semistable.
\end{conjecture}

On the other hand, the conjecture does not hold if one allows for non-integrable distributions. We provide in the Appendix a generalization of Ottaviani's Bertini type Theorem \cite[Teorema 2.8]{O} to produce examples of codimension one distributions on $\p3$; as an application, we exhibit a non-integrable codimension one distribution on $\p3$ of degree 14 whose tangent sheaf is indecomposable, locally free, and not $\mu$-semistable, see Example \ref{exeNonsemi}.

\subsection*{Acknowledgments}
We are grateful to Alex Massarenti for interesting conversations. The authors also thank the anonymous referees for giving many suggestions that helped improving the presentation of the paper.  CA was partially supported by the FAPESP grant number 2014/23594-6.  
MC was partially supported by CNPq, CAPES and the FAPESP grant number 2015/20841-5; he is grateful to University of Oxford for its hospitality. Both CA and MC visited the University of Campinas during the completion of this work in 2015 and 2016, and are also grateful for its hospitality.
MJ was partially supported by the CNPq grant number 303332/2014-0 and the FAPESP grants number 2014/14743-8 and 2016/03759-6; he thanks the University of Edinburgh for a pleasant and productive visit in 2017.

%%%%%%%%%%%%%%%%%%%%%%%%%%%%%%%%%%%%%%%%%%%%%%%%%%%%%%%%%%%%%%%%%%%%%%%%%%%%
%%%%%%%%%%%%%%%%%%%%%%%%%%%%%%%%%%%%%%%%%%%%%%%%%%%%%%%%%%%%%%%%%%%%%%%%%%%%

\section{Distributions and their moduli spaces}

We begin by setting up the notation and nomenclature to be used in the rest of the paper.

\subsection{Distributions and foliations}\label{dist+fol}

A \emph{codimension $r$ distribution} $\sF$ on a smooth complex manifold $X$ is given by an exact sequence
\begin{equation}\label{eq:Dist}
\mathscr{F}:\  0  \longrightarrow T_\sF \stackrel{\phi}{ \longrightarrow} TX \stackrel{\pi}{ \longrightarrow} N_{\sF}  \longrightarrow 0,
\end{equation}
where $T_\sF$ is a coherent sheaf of rank $s:=\dim(X)-r$, and $N_{\sF}$ is a torsion free sheaf. The sheaves $T_\sF$ and $N_{\sF}$ are called the \emph{tangent} and the \emph{normal} sheaves of $\mathscr{F}$, respectively. Note that $T_\sF$ must be reflexive \cite[Proposition 1.1]{H2}.

Two distributions $\mathscr{F}$ and ${\mathscr{F}}'$ are said to be isomorphic if there exists an isomorphism $\beta:T_{\mathscr{F}}\to T_{\mathscr{F}'}$ such that $\phi'\circ\beta=\phi$. By taking exterior power $\wedge^{n-r}\phi :\det(T_{\mathscr{F}})\to\wedge^{n-r}TX$, every codimension $r$ distribution induces a section of $H^0(X,\wedge^{n-k}TX \otimes \det(T_\sF)^{\vee})$. If $\mathscr{F}$ and ${\mathscr{F}}'$ are isomorphic distributions, $\wedge^{n-r}\phi'\circ\det(\beta)=(\wedge^{n-r}\phi)$; but $\det(\beta)=\lambda$ for some constant $\lambda\in\C^*$ since $\det(T_\sF)\simeq \det(T_{\sF'})$, i.e, $\wedge^{n-r}\phi'=\lambda \cdot \wedge^{n-r}\phi$. Therefore, every isomorphism class of codimension $r$ distribution on $X$ induces an element in the projective space
$$
\mathbb{P}H^0(\wedge^{n-r}TX \otimes \det(T_\sF)^{\vee}) = 
\mathbb{P}H^0(\Omega^r_X \otimes\det(TX)\otimes \det(T_\sF)^{\vee}). 
$$

% \begin{Obs}\label{equiv-distributions}
% If $(F,\phi)\sim(F',\phi')$, we have by definition that   $\beta_s:F_s\to F_s'$  are  isomorphism for all $s\in S$. By taking exterior power $ \wedge^{n-k}\phi_s:\det(F_s)\to \wedge^{n-k}TX$ and $\wedge^{n-k}\phi_s':\det(F_s)\to \wedge^{n-k} TX$, where $k$ is the codimension of $(F,\phi)$ and $(F',\phi')$.   
% In particular, $\wedge^{n-k}\phi_s'=(\wedge^{n-k}\phi_s)\circ \det(\beta_s)$, but $\det(\beta_s)=\lambda$ is constant since $\det(F_s)\simeq \det(F_s')$, i.e, $\wedge^{n-k}\phi_s'=\lambda \cdot \wedge^{n-k}\phi_s$.  Therefore, we have that isomorphic distributions induce  the same element in the projective space
% $$
% \mathbb{P}H^0(X,\wedge^{n-k}TX \otimes \det(F_s)^{\vee} ). 
% $$
% \end{Obs}

The \emph{singular scheme} of $\mathscr{F}$ is defined as follows. Taking the maximal exterior power of the dual morphism $\phi^\vee:\Omega^1_X\to T_\sF^\vee$ we obtain a morphism $\Omega^{s}_X\to \det(T_\sF)^\vee$; the image of such morphism is the ideal sheaf $I_{Z/X}$ of a subscheme $Z\subset X$, the singular scheme of $\mathscr{F}$, twisted by $\det(T_\sF)^\vee$.

Note that if the tangent sheaf $T_\sF$ is locally free, $Z$ coincides, as a set, with the singular set of the normal sheaf. Indeed, by definition
$$ \sing(N_{\sF}) := \bigcup_{p=1}^{\dim(X)-1} \supp\left( \inext^p(N_{\sF},\OX) \right), $$
where by $\supp(-)$ we mean the set-theoretical support of a sheaf. When $T_\sF$ is locally free, then 
$\inext^p(N_{\sF},\OX)=0$ for $p\ge2$, and 
$$ \sing(N_{\sF}) = \supp\left( \inext^1(N_{\sF},\OX) \right) = \left\{ ~ x \in X ~|~ \phi(x) ~~\textrm{is not injective} ~ \right\} = \supp(\OZ). $$     

A \emph{sub-distribution} of $\mathscr{F}$ is a distribution
$$ \mathscr{G}:\  0  \longrightarrow T_\sG \stackrel{\phi}{ \longrightarrow} TX\to N_{\sG}  \longrightarrow 0, $$
whose tangent sheaf $T_\sG$ is a subsheaf of $T_\sF$. The quotient $N_{\sG/\sF}:=T_\sF/T_\sG$ is called the \emph{relative normal sheaf}. In addition, the following exact sequence is satisfied:
\begin{equation} \label{sqc relat normal}
0 \to N_{\sG/\sF} \to N_{\sG} \to N_{\sF} \to 0. 
\end{equation}
It follows that $N_{\sG/\sF}$ is also torsion free.

Finally, we introduce the notion of integrability. A \emph{foliation} is an integrable distribution, that is a distribution
$$ \mathscr{F}:\  0  \longrightarrow T_\sF \stackrel{\phi}{ \longrightarrow} TX \stackrel{\pi}{ \longrightarrow} N_{\sF}  \longrightarrow 0 $$
whose tangent sheaf is closed under the Lie bracket of vector fields, i. e. $[\phi(T_\sF),\phi(T_\sF)]\subset \phi(T_\sF)$.  Clearly, every distribution of codimension $\dim(X)-1$ is integrable. A \emph{sub-foliation} of $T_\sF$ is an integrable sub-distribution.

When $r=1$, the normal sheaf, being a torsion free sheaf of rank $1$, must be the twisted ideal sheaf $I_{Z/X}\otimes \det(TX)\otimes\det(T_\sF)^\vee$ of a closed subscheme $Z\subset X$ of codimension at least 2, which is precisely the singular scheme of $\mathscr{F}$. Let $\mathscr{U}$ be the maximal subsheaf of $\OZX$ of codimension $>2$ (cf. \cite[Definition 1.1.4]{HL}), so that one has an exact sequence of the form
\begin{equation}\label{OZ->OC}
0 \to \mathscr{U} \to \OZX \to \OCX \to 0
\end{equation}
where $C\subset X$ is a (possibly empty) subscheme of pure codimension 2; in particular, it follows that $\inext^{p}(\OC,\OX)=0$ for $p>2$, see \cite[Proposition 1.1.10]{HL}.
Note also that $\mathscr{U}\simeq I_{C/Z}$ . 

In addition, for each $p\ge 1$ and setting $\mathscr{L}:=\det(TX)\otimes\det(T_\sF)^\vee$, we obtain:
\begin{equation} \label{display exts}
\inext^p(T_\sF,\OX)\simeq \inext^{p+1}(I_{Z/X},\OX)\otimes\mathscr{L}^\vee \simeq
\inext^{p+2}(\OZX,\OX)\otimes\mathscr{L}^\vee \simeq \inext^{p+2}(\mathscr{U},\OX)\otimes\mathscr{L}^\vee .
\end{equation}
As an immediate consequence we obtain the following generalization of \cite[Theorem 3.2]{GP}.

\begin{Lemma}\label{loc free tg}
The tangent sheaf of a codimension one distribution is locally free if and only if its singular locus has pure codimension 2.
\end{Lemma}
\begin{proof}
On the one hand, $T_\sF$ is locally free if and only if $\inext^p(T_\sF,\OX)=0$ for each $p\ge1$. On the other hand, since $\mathscr{U}$ is a sheaf of codimension at least 3, $\inext^q(\mathscr{U},\OX)=0$ for $q\ge3$ if and only if $\mathscr{U}=0$, which is the same as saying that $Z$ has pure codimension 2, see again \cite[Proposition 1.1.10]{HL}.
\end{proof}

\begin{Obs}\rm
The singular scheme of an integrable distribution of codimension one on $X=\P^n$ is always nonempty, and has components of codimension 2
\cite{J}. Giraldo and Pan-Collantes proved Lemma \ref{loc free tg} for the case $X=\p3$ using a completely different argument, cf. \cite[Theorem 3.2]{GP}. We also note that Lemma \ref{loc free tg} implies that the hypotheses of \cite[Theorem 1 and Theorem 2]{CJV} are redundant. 
\end{Obs}

%--------------------------------------------------------

\subsection{The Quot-scheme of TX}\label{secQuot}

Let us now recall the definition of the Grothendieck's Quot-scheme for the tangent bundle $TX$ and some of its basic properties. We use \cite[Section 2.2]{HL} as main reference. $X$ now denotes a polarized nonsingular projective variety of dimension $n$. 

Let $\sch$ denote the category of schemes of finite type over $\C$, and $\sets$ be the category of sets. Fix a polynomial $P\in\Q[t]$, and consider the functor
$$ \quotf^P : \sch^{\rm op} \to \sets ~~,~~ \quotf^P(S) := \{(N,\eta)\}/\sim $$
where
\begin{itemize}
\item[(i)] $N$ is a coherent sheaf of $\mathcal{O}_{X\times S}$-modules, flat over $S$, such that the Hilbert polynomial of $N_s:=N|_{X\times\{s\}}$ is equal to $P$ for every $s\in S$;
\item[(ii)] $\eta:\pi^*_X TX \to N$ is an epimorphism, where $\pi_X:X\times S\to X$ is the standard projection onto the first factor.
\end{itemize}
In addition, we say that $(N,\eta)\sim(N',\eta')$ if there exists an isomorphism $\gamma:N\to N'$ such that $\gamma\circ\eta=\eta'$.

Finally, if $f:R\to S$ is a morphism in $\sch$, we define $\quotf^P(f):\quotf^P(S)\to\quotf^P(R)$ by
$(N,\eta) \mapsto (f^*N,f^*\eta)$. Elements of the set $\quotf^P(S)$ will be denoted by $[N,\eta]$.

\begin{Theorem}\label{quot-thm}
The functor $\quotf^P$ is represented by a projective scheme $\mathcal{Q}^P$ of finite type over $\C$, that is, there exists an isomorphism of functors $\quotf^P \stackrel{\sim}{\longrightarrow} \Hom(\cdot,\mathcal{Q}^P)$. 
%In addition, if \linebreak $\Ext^1(\ker\eta,N)=0$, then $\mathcal{Q}^P$ is smooth at a point $[N,\eta]$, and $\dim T_{[N,\eta]}\mathcal{Q}^P = \dim\Hom(\ker\eta,N)$.
\end{Theorem}

In what follows, it will be important to consider the following subset of the Quot-scheme:
\begin{equation} \label{d^p}
\mathcal{D}^P := \{ [N,\eta]\in \mathcal{Q}^{P_{TX}-P} ~|~ N ~\textrm{is torsion free}\},
\end{equation}
where $P_{TX}$ is the Hilbert polynomial of the tangent bundle. Note that, by \cite[Proposition 2.3.1]{HL}, $\mathcal{D}^P$ is an open subset of $\mathcal{Q}^{P_{TX}-P}$.

%--------------------------------------------------------

\subsection{Moduli spaces of distributions}

Fix a polynomial $P\in\Q[t]$, and consider the functor
$$ \dist^P : \sch^{\rm op} \to \sets ~~,~~ \dist^P(S) := \{(F,\phi)\}/\sim $$
where
\begin{itemize}
\item[(i)] $F$ is a coherent sheaf of $\mathcal{O}_{X\times S}$-modules, flat over $S$, such that the Hilbert polynomial of $F_s:=F|_{X\times\{s\}}$ is equal to $P$ for every $s\in S$;
\item[(ii)] $\phi:F \to \pi^*_X TX$ is a morphism of sheaves such that $\phi_s:F_s\to TX$ is injective, and $\coker\phi_s$ is torsion free for every $s\in S$. 
\end{itemize}

In addition, we say that $(F,\phi)\sim(F',\phi')$ if there exists and isomorphism $\beta:F\to F'$ such that $\phi'\circ\beta=\phi$.

Finally, if $f:R\to S$ is a morphism in $\sch$, we define $\dist^P(f):\dist^P(S)\to\dist^P(R)$ by
$(F,\phi) \mapsto (f^*F,f^*\phi)$. Equivalence classes in the set $\dist^P(S)$ will be denoted by $[F,\phi]$.

Each pair $(F,\phi)$ should be regarded as a \emph{family of distributions on $X$ parametrized by the scheme $S$} (or an $S$-family, for short): for each $s\in S$, we have the distribution given by the short exact sequence
$$ \sF_s ~:~ 0 \to F_s \stackrel{\phi_s}{\longrightarrow} TX \stackrel{\eta_{\phi_s}}{\longrightarrow} \coker\phi_s \to 0 , $$
where $\eta_\phi$ denotes the canonical projection $\pi_X^*TX\to\coker\phi$. Note that if $(F,\phi)$ and $(F',\phi')$ are equivalent $S$-families of distributions on $X$, then for each $s$, the distributions $\sF_s$ and $\sF_s'$ are isomorphic.

\begin{Prop}\label{dist-prop}
The functor $\dist^P$ is represented by the quasi-projective scheme $\mathcal{D}^P$, described in display \eqref{d^p}, that is, there exists an isomorphism of functors
$\dist^P \stackrel{\sim}{\longrightarrow} \Hom(\cdot,\mathcal{D}^P)$.
\end{Prop}

\begin{proof}
Let $S$ be a scheme of finite type over $\C$, and take $[F,\phi]\in\dist^P(S)$. It follows that $[\coker\phi,\eta_\phi]$ is an element of $\quotf^{P_{TX}-P}(S)$. Since $\quotf^{P_{TX}-P}(S)\simeq\Hom(S,\mathcal{Q}^{P_{TX}-P})$, let $g_{[F,\phi]}:S\to \mathcal{Q}^{P_{TX}-P}$ be the morphism corresponding to the $S$-family $[\coker\phi,\eta_\phi]$. Note that $\img(g_{[F,\phi]})\subset \mathcal{D}^P$, since $\coker\phi_s$ is torsion free for every $s\in S$, thus actually $g_{[F,\phi]}\in\Hom(S,\mathcal{D}^{P})$. 

We have therefore obtained a morphism of sets $\dist^P(S)\to\Hom(S,\mathcal{D}^{P})$ for each scheme $S$; we check that it is bijective. Indeed, we can compose a morphism $g:S\to \mathcal{D}^P$ with the inclusion of $\mathcal{D}^P$ into $\mathcal{Q}^{P_{TX}-P}$ to obtain an element of $\Hom(S,\mathcal{Q}^{P_{TX}-P})$, which corresponds to a $S$-family
$[N,\eta]\in\quotf^{P_{TX}-P}(S)$; since $g(s)\in \mathcal{D}^P$ for each $s\in S$, it follows that $[\ker\eta,\iota_\eta]$ belongs to $\dist^P(S)$. Moreover, we also obtained a commutative diagram
\begin{equation}\label{dist-quot-diagram}
\xymatrix{
\dist^P(S) \ar@{^{(}->}[d] \ar[r]^{\sim} & \Hom(S,\mathcal{D}^P) \ar@{^{(}->}[d] \\
\quotf^{P_{TX}-P}(S) \ar[r]^{\sim} & \Hom(S,\mathcal{Q}^{P_{TX}-P})
}\end{equation}
in which the vertical arrows are injective, and the horizontal arrows are bijective.

To conclude the proof, we must check that the bijection between $\dist^P(S)$ and $\Hom(S,\mathcal{D}^{P})$ constructed above underlies a natural transformation between the functors $\dist^P$ and $\Hom(\cdot,\mathcal{D}^P)$. Indeed, the representability of the Quot functor $\quotf^{P_{TX}-P}$ yields the following commutative diagram, for every morphism $f\in\Hom(R,S)$:
\begin{equation}\label{quot-diagram}
\xymatrix{
\quotf^{P_{TX}-P}(S) \ar[d]^{f^*} \ar[r]^{\sim} & \Hom(S,\mathcal{Q}^{P_{TX}-P}) \ar[d]^{\circ f} \\
\quotf^{P_{TX}-P}(R) \ar[r]^{\sim} & \Hom(R,\mathcal{Q}^{P_{TX}-P})
}\end{equation}
This last diagram together with diagrams of the form (\ref{dist-quot-diagram}) for the schemes $R$ and $S$ yields the desired diagram:
$$ \xymatrix{
\dist^{P}(S) \ar[d]^{f^*} \ar[r]^{\sim} & \Hom(S,\mathcal{D}^{P}) \ar[d]^{\circ f} \\
\quotf^{P}(R) \ar[r]^{\sim} & \Hom(R,\mathcal{D}^{P})
} $$
\end{proof}

Next, we recall the construction of the Gieseker--Maruyama moduli space of semistable sheaves on $X$; we use \cite[Section 4]{HL} as our main reference. Recall that a torsion free sheaf $\mathscr{E}$ on $X$ is said to be (semi)stable if every proper nontrivial subsheaf $\mathscr{F}\subset\mathscr{E}$ satisfies
$$ \frac{P_\mathscr{F}(t)}{\rk\mathscr{F}} (\le)< \frac{P_\mathscr{E}(t)}{\rk\mathscr{E}} $$
for $t$ sufficiently large. Consider the functor 
$$ \mathcal{M}^P : \sch^{\rm op} \to \sets ~~,~~ \mathcal{M}^P(S) := \{F\}/\sim $$
where $F$ is a sheaf on $X\times S$ flat over $S$ such that for each $s\in S$, the sheaf $F_s:=F|_{X\times\{s\}}$ is semistable and has Hilbert polynomial equal to $P$; the equivalence relation $\sim$ is just isomorphism of sheaves on $X\times S$.

As it is well known, $\mathcal{M}^P$ admits a \emph{coarse moduli space}, here denoted by $M^P$, which is a projective scheme; this means that the closed points of $M^P$ are bijective with the S-equivalence classes of semistable sheaves on $X$ with fixed Hilbert polynomial $P$, and that there exists a natural transformation $\mathcal{M}^P \to \Hom(\cdot,M^P)$. Below, $M^{P,st}$ will denote the open subset of $M^P$ consisting of stable sheaves.

Let $\dist^{P,ss}$ be the subfunctor of $\dist^P$ given by 
$$ \dist^{P,ss} : \sch^{\rm op} \to \sets ~~,~~ \dist^{P,ss}(S) := \{(F,\phi)\}/\sim $$
where $F_s$ is now assumed to be semistable for each $s\in S$. Clearly, $\dist^{P,ss}\simeq \Hom(\cdot,\mathcal{D}^{P,ss})$, where 
$$ \mathcal{D}^{P,ss} := \{ [F,\phi]\in \mathcal{D}^P ~|~ F \textrm{ is semistable } \}; $$ 
note that $\mathcal{D}^{P,ss}$ is an open subset of $\mathcal{D}^P$. Similarly, we will also consider the following open subset of $\mathcal{D}^P$:
$$ \mathcal{D}^{P,st} := \{ [F,\phi]\in \mathcal{D}^P ~|~ F \textrm{ is stable } \}. $$ 

\begin{Lemma} \label{forget phi}
There exists a forgetful morphism
$$ \varpi :  \mathcal{D}^{P,ss} \to M^P ~~,~~ \varpi([\sF]) := [T_\sF] . $$
In addition, if $T_{\sF}$ is stable and satisfies $\Ext^1(T_{\sF},TX)=\Ext^2(T_{\sF},T_{\sF})=0$, then $[\sF]$ is a nonsingular point of ${\mathcal D}^{P,st}$, $\varpi$ is a submersion at $[T_\sF]\in M^{P,st}$ and 
$$ \dim_{[\sF]} {\mathcal D}^{P,st} = \dim \Ext^1(T_{\sF},T_{\sF}) + \dim\Hom(T_\sF,TX) - 1. $$
\end{Lemma}

\begin{proof}
Consider the following forgetful natural transformation $\Pi: \dist^{P,ss} \to \mathcal{M}^P$ given by
$$ \Pi_S : \dist^{P,ss}(S) \to \mathcal{M}^P(S) ~~,~~ \Pi_S([F,\phi]):=[F]. $$
Composing this with the isomorphism $\Hom(\cdot,\mathcal{D}^{P,ss})\simeq \dist^{P,ss}$, on one hand, and with the natural transformation $\mathcal{M}^P \to \Hom(\cdot,M^P)$, on the other, we obtain a natural transformation $\Pi:\Hom(\cdot,\mathcal{D}^{P,ss})\to \Hom(\cdot,M^P)$. The desired morphism is given by
$\varpi:=\Pi_{\mathcal{D}^{P,ss}}(1_{\mathcal{D}^{P,ss}})$.

Applying the functor $\Hom(T_{\sF},\cdot)$ to the exact sequence in display \eqref{eq:Dist} we obtain 
$$ \cdots \to \Ext^1(T_{\sF},TX) \to \Ext^1(T_{\sF},N_{\sF}) \to \Ext^2(T_{\sF},T_{\sF}) \to \cdots $$
so our hypotheses imply that $\Ext^1(T_{\sF},N_{\sF})=0$. According to \cite[Proposition 2.2.8]{HL}, it follows that $[\sF]$ is a nonsingular point of the Quot-scheme $\mathcal{Q}^P$, thus also of ${\mathcal D}^{P,st}$. Furthermore, the tangent space of $\mathcal{Q}^P$ at $[\sF]$ is precisely $\Hom(T_{\sF},N_{\sF})$; considering the exact sequence
$$ 0 \to \Hom(T_{\sF},T_{\sF}) \to \Hom(T_{\sF},TX) \to \Hom(T_{\sF},N_{\sF}) \to \Ext^1(T_{\sF},T_{\sF}) \to 0, $$
we conclude that $\varpi$ is a submersion at $[T_\sF]$ (since $\Ext^1(T_{\sF},T_{\sF})$ is the tangent space of $M^{P,st}$ at $[T_\sF]$). The formula for the dimension of ${\mathcal D}^{P,st}$ at $[\sF]$ follows immediately from the exact sequence in display just above and the fact that $\dim\Hom(T_{\sF},T_{\sF})=1$ since $T_{\sF}$ is stable.
\end{proof}

We observe that the fibre over a point $[F]\in M^{P,ss}$ is precisely the set of all distributions whose tangent sheaf is isomorphic to a given sheaf:
$$ {\mathcal D}(F) := 
\{ \phi\in\mathbb{P}\Hom(F,TX) ~~|~~ \ker\phi=0 ~\textrm{and}~ \coker\phi ~\textrm{is torsion free} ~\}, $$
which is an open subset of $\mathbb{P}\Hom(F,TX)$. Therefore if the image of $\varpi$ is irreducible and \linebreak $\dim\Hom(F,TX)$ is constant for all $[F]$ in the image of $\varpi$, then it follows that $\mathcal{D}^{P,ss}$ is also irreducible. Moreover, since the tangent sheaf of a distribution is always reflexive, the image of $\varpi$ is contained in the open subset of $M^P$ consisting of reflexive sheaves. These remarks lead us to the following statement; let $M^{P,r,st}$ denote the open subset of $M^P$ consisting of stable reflexive sheaves.

\begin{Lemma} \label{forget phi 2}
Assume that the forgetful morphism $\varpi:\mathcal{D}^{P,st} \to M^{P,r,st}$ is surjective, and that $M^{P,r,st}$ is irreducible. If $\dim\Hom(F,TX)$ is constant for all $[F]\in M^{P,r,st}$, then $\mathcal{D}^{P,st}$ is irreducible and
$$ \dim \mathcal{D}^{P,st} = \dim M^{P,r,st} + \dim\Hom(F,TX) - 1 . $$
If, in addition, $\Ext^2(T_{\sF},T_{\sF})=0$ for every $[\sF]\in\mathcal{D}^{P,st}$, then $\mathcal{D}^{P,st}$ is nonsingular and 
$$ \dim {\mathcal D}^{P,st} = \dim\Ext^1(T_{\sF},T_{\sF}) + \dim\Hom(T_\sF,TX) - 1. $$
\end{Lemma}

\begin{proof}
The first claim is a well-known consequence of the theorem on the dimension of the fibers; see for instance \cite[Section 6.3, Theorems 7 and 8]{Shafa}.

Since $\varpi:\mathcal{D}^{P,st} \to M^{P,r,st}$ is surjective, the hypothesis $\Ext^2(T_{\sF},T_{\sF})=0$ for every $[\sF]\in\mathcal{D}^{P,st}$ imply that $ M^{P,r,st}$ is nonsingular, and $\dim M^{P,r,st}=\dim\Ext^1(T_{\sF},T_{\sF})$. 

To see that ${\mathcal D}^{P,st}$ is nonsigular, we must check that the forgetful morphism $\varpi$ is a submersion. Applying the functor $\Hom(T_\sF,\cdot)$ to the exact sequence in display \eqref{eq:Dist} we obtain
\begin{equation}\label{x.x.x}
\begin{aligned}
0 \to & \Hom(T_\sF,T_\sF) \to \Hom(T_\sF,TX) \to \Hom(T_\sF,N_\sF) \to \\
& \Ext^1(T_\sF,T_\sF) \to \Ext^1(T_\sF,TX) \to \Ext^1(T_\sF,N_\sF) \to 0
\end{aligned}
\end{equation}
since $\Ext^2(T_\sF,T_\sF)=0$. Note that $\Hom(T_\sF,N_\sF)$ is the tangent space of ${\mathcal D}^{P,st}$ at the point $[\sF]$, while $\Ext^1(T_\sF,T_\sF)$ is the tangent space of $ M^{P,r,st}$ at the point $[T_\sF]$; so counting dimension in the exact sequence in display \eqref{x.x.x}, we conclude that $\dim\Ext^1(T_\sF,TX)=\dim \Ext^1(T_\sF,N_\sF)$. It follows that the epimorphism in the second line of display \eqref{x.x.x} must be an isomorphism, thus the morphism $\Hom(T_\sF,N_\sF) \to \Ext^1(T_\sF,T_\sF)$ must also be an epimorphism, as desired.
\end{proof}

These last two lemmas will be used in Section \ref{sec:modspc} below to provide a precise description of the moduli space of codimension one distributions on $X=\p3$ for several choices of Hilbert polynomials.

Similarly, one can also define a moduli functor for foliations on $X$, assigning to each scheme $S$ the set of $S$-families of integrable distributions on $X$ whose tangent sheaves have fixed Hilbert polynomial $P$. Such a scheme is represented by a closed subscheme $\mathcal{F}^P$ of $\mathcal{D}^P$, or a locally closed subscheme of $Q^{P_{TX}-P}$, cf. \cite[Section 6]{Q}.

%%%%%%%%%%%%%%%%%%%%%%%%%%%%%%%%%%%%%%%%%%%%%%%%%%%%%%%%%%%%%%%%%%%%%%%%%%%%
%%%%%%%%%%%%%%%%%%%%%%%%%%%%%%%%%%%%%%%%%%%%%%%%%%%%%%%%%%%%%%%%%%%%%%%%%%%%

\section{Codimension one distributions on $\p3$} \label{ss:Sing}

From now on, we will only consider codimension one distributions on $X=\P^3$. In this case, the integer $d:=2-c_1(T_\sF)\geq0$ is called the \emph{degree} of the distribution $\sF$, and $N_{\sF}=I_{Z/\P^3}(d+2)$ where $Z$ is the singular scheme of $\sF$. Therefore, sequence (\ref{eq:Dist}) now reads
\begin{equation}\label{eq:Dist p3}
\mathscr{F}:\  0  \longrightarrow T_\sF \stackrel{\phi}{ \longrightarrow} T\p3 \stackrel{\pi}{ \longrightarrow}
I_{Z/\P^3}(d+2)  \longrightarrow 0,
\end{equation} 
where $T_\sF$ is a rank 2 reflexive sheaf.

As pointed out in Section \ref{dist+fol}, a codimension one distribution of degree $d$ on $\p3$ can also be represented by a %class of
section $\o\in H^0 (\Om(d+2))$, given by the dual of the morphism $\pi:T\p3\to I_{Z/\P^3}(d+2)$. On the other hand, such section yields a sheaf map $\o:\O \rightarrow\Om(d+2)$; taking duals, we get a cosection $\o^{\vee}:\Om(d+2))^{\vee}=T\P^3(-(d+2))\rightarrow\O$ whose image is the ideal sheaf $I_{Z/\P^3}$ of the singular scheme. The kernel of $\o^{\vee}$ is the tangent sheaf $\sF$ of the distribution twisted by $\O(-(d+2))$. From this point of view, the integrability condition is equivalent to $\o\wedge d\o=0$.

The 1-form $\o$ can be written down in homogeneous coordinates
\[ \o=\sum_{i=0}^3 A_i dz_i, \quad A_i\in H^0(\O(d+1)), \]
where $[z_0:z_1:z_2:z_3]$ homogeneous coordinates of $\P^3$; in addition, the coefficients $A_i$ must satisfy the condition
$$ i_R \omega = \sum_{i=0}^3 z_iA_i = 0, $$
where $i_R$ denotes the inner product of the vector field $R$ with differential forms. 

This information can also be packaged in a complex of sheaves
\begin{equation}\label{eq:monad}
\sF_{\bullet}: \O\xrightarrow{R}\O(1)^{\oplus4}\xrightarrow{\boldsymbol{\o}}\O(d+2),\quad R =\begin{pmatrix}
z_0 \\z_1\\ z_2 \\ z_3
\end{pmatrix}
\quad
\boldsymbol{\o}=(A_0,A_1,A_2,A_3)
\end{equation}
such that
\[
\begin{cases} \mathscr{H}^0=0 & \\
              \mathscr{H}^1=T_\sF \quad &\mbox{tangent sheaf}\equiv \ker(\omega)\\
              \mathscr{H}^2=I_{Z/\P^3}(d+2)\quad & \mbox{normal sheaf}\equiv \coker(\omega)
\end{cases}
\]
The complex $\sF_{\bullet}$ provides an alternative definition for codimension one distributions.

Since the Koszul complex associated to $R$ is exact, we also conclude that there exists an antisymmetric polynomial matrix $B$, whose entries are polynomials of degree $d$, such that
$$
(A_0,A_1,A_2,A_3)^t= B \cdot (z_0,z_1,z_2,z_3)^t.
$$
Therefore, a distribution induces a rational map
$$
\begin{array}{cccc}
& \P^3  &  \dashrightarrow & \mathbb{P}(\Omega_{\P^3}^1) \subset \P^3 \times (\P^3 )^* \\
\\
 & (z_0:z_1:z_2:z_3) & \longmapsto & ((z_0:z_1:z_2:z_3), B \cdot (z_0:z_1:z_2:z_3)^t)
\end{array}
$$
providing yet another point of view for codimension one distributions on $\p3$.

% We recall from Remark \ref{equiv-distributions} that 
% two isomorphic codimension one distributions $\sF$ and $\sF'$ on $\p3$   correspond   to a point $[\o]\in \P H^0(\wedge^2T\P^3\otimes\det(T_{\sF})\otimes K_{\P^3}^{\vee} )=\P H^0 (\Om(d+2))$.

%--------------------------------------------------------

\subsection{Numerical invariants of the singular locus}

Let $\sF$ be a codimension one distribution on $X=\p3$ given as in the exact sequence (\ref{eq:Dist p3}), with tangent sheaf $T_\sF$ and singular scheme $Z$. As in the exact sequence (\ref{OZ->OC}), let $\mathscr{U}$ denote the maximal 0-dimensional subsheaf of $\OZ$, so that the quotient sheaf is the structure sheaf of a subscheme $C\subset Z\subset \p3$ of pure dimension 1. We will say that $C$ is the \emph{1-dimensional component of $\sing(\sF)$}, while $\supp(\mathscr{U})$ is the \emph{0-dimensional component of $\sing(\sF)$}.

%We remark that, in general, $\mathscr{U}$ might not be the structure sheaf of a 0-dimensional subscheme of $\p3$, and $C$ might be reducible and/or non-reduced. However, if, for instance, the support of $\mathscr{U}$ is disjoint from $C$, then $\mathscr{U}$ is the structure sheaf of a 0-dimensional scheme $R\subset\p3$, and $\OZ=\OC+{\mathcal O}_{R/\p3}$.

\begin{Theorem}\label{P:SLocus}
In the notation of the previous paragraph, one has
\begin{equation}\label{eq:length}
\begin{array}{rcl}
c_2(T_\sF)&=& d^2+2-\deg(C) \\
c_3(T_\sF)={\rm length}(\mathscr{U})&=&d^3+2d^2+2d-\deg(C)\cdot(3d-2)+2p_a(C)-2
\end{array}
\end{equation}
%\begin{equation}\label{eq:length}
%{\rm length}(R)=d^3+2d^2+2d-deg(C)(3d-2)+2p_a(C)-2
%\end{equation}
where $p_a(C)$ denotes the arithmetic genus of $C$.
\end{Theorem}
\begin{proof}
Considering the exact sequence (\ref{eq:Dist p3}), use $c(T\mathbb{P}^3)=c(T_\sF)\cdot c(I_{Z/\P^3}(d+2))$ to obtain
\begin{equation}\label{eq:IT}
\begin{array}{lll}
4 & = & c_1(T_\sF)+c_1(I_{Z/\P^3}(d+2))\\
6 & = & c_1(T_\sF)\cdot c_1(I_{Z/\P^3}(d+2))+c_2(T_\sF)+c_2(I_{Z/\P^3}(d+2))\\
4 & = & c_3(T_\sF) + c_3(I_{Z/\P^3}(d+2)) + c_1(T_\sF)\cdot c_2(I_{Z/\P^3}(d+2)) + c_2(T_\sF)\cdot c_1(I_{Z/\P^3}(d+2))
\end{array}
\end{equation}

The first equation gives $c_1(T_\sF)=2-d$, as pointed out in the previous section. The exact sequence (\ref{OZ->OC}) is equivalent to the sequence
\begin{equation}\label{IZ->IC}
0 \to I_{Z/\P^3} \to I_{C/\P^3} \to \mathscr{U} \to 0
\end{equation}
It follows that $c_2(I_{Z/\P^3}(d+2))=c_2(I_{C/\P^3}(d+2))=\deg(C)$, thus substitution into the second equation yields
$$ c_2(T_\sF)= d^2+2-\deg(C). $$

Moreover, substituting the expressions for the first and second Chern classes into the third equation we obtain
\begin{equation} \label{star}
d^3 + 2d^2 + 2d + c_3(I_{Z/\P^3}(d+2)) + c_3(T_\sF) - 2d\cdot\deg(C)=0.
\end{equation}
Note that
\begin{equation}\label{ast}
c_3(I_{Z/\P^3}(d+2))=c_3(I_{Z/\P^3})-(d+2)\cdot\deg(C),
\end{equation}
while it follows from the sequence (\ref{IZ->IC}) that
\begin{equation}\label{ast ast}
c_3(I_{Z/\P^3})=c_3(\IC)-c_3(\mathscr{U}) = 2  p_a(C) - 2 + 4\cdot\deg(C) - 2\cdot\mathrm{length}(\mathscr{U}).
\end{equation}
Finally, note that 
$$ \inext^1(T_\sF,\op3) \simeq \inext^2(I_{Z/\P^3}(d+2),\op3) \simeq \inext^3(\OZ(d+2),\op3)
\simeq \inext^3(\mathscr{U},\op3) .$$
It follows that 
\begin{equation}\label{c3}
c_3(T_\sF)=\mathrm{length}(\inext^1(T_\sF,\op3))=\mathrm{length}(\mathscr{U});
\end{equation}
indeed, the first equality is given in \cite[Proposition 2.6]{H2}, while the second one is obtained using the spectral sequence for local to global Ext's and Serre duality
$$ \Ext^3(\mathscr{U},\O) \simeq H^0(\inext^3(\mathscr{U},\op3))  \simeq H^0(\mathscr{U})^*, $$
respectively.

Substituting (\ref{ast}), (\ref{ast ast}), and (\ref{c3}) into the equation (\ref{star}) we obtain
$$ \mathrm{length}(\mathscr{U})=d^3+2d^2+2d-(3d-2)\deg(C)+2p_a(C)-2 $$
as claimed.
\end{proof}

By \cite[Theorem 2.3 page 87]{J}, it follows that the distribution corresponding to a generic section $\o\in H^0(\Om(d+2))$ has only isolated singularities. In this case $\OZ=\mathscr{U}$, and we also have
\begin{eqnarray}
\label{eq:c2} c_2(T_\sF) & = & d^2 + 2 ,~~{\rm and} \\
\label{eq:ising} c_3(T_\sF) & = & \mathrm{length}(Z) = d^3 + 2d^2 + 2d = c_3(\Om(d+2)).
\end{eqnarray}

If $T_\sF$ is locally free, then $Z=C$ by Lemma \ref{loc free tg}, and one obtains the following relations between the arithmetic genus, Euler characteristic, and degree of $C$: 
\begin{eqnarray}
\label{eq:gen} p_a(C) & = & d^3 - 2d^2 + 2d -\frac{1}{2}(3d-2)(d^2+2-\deg(C)) - 1 \\
\label{euler}
\chi(\OC) & = & 1 - p_a(C) = \frac{1}{2}\left( d^3+2d^2+2d-(3d-2)\deg(C) \right)
\end{eqnarray}

\begin{Obs}\rm
In \cite{Vai}, the author determines the number of isolated singularities under the hypothesis that $\supp(\mathscr{U})$ is disjoint from $C$.
\end{Obs}

The previous equalities impose a rather strong condition on the quantities  involved, since they are all integer numbers. We explore this fact in the next corollaries, using the software Mathematica \cite{Math}  to conclude that certain polynomial equations have no integer solutions.

\begin{Cor}
Let $\sF$ be a codimension one distribution of degree $d$ with locally free tangent sheaf. If  $d\geq 5$, then $C$ can neither be canonical, nor a nonsingular rational curve.
\end{Cor}
\begin{proof}
If either $p_a(C)=0$ or $2p_a(C)-2=\deg(C)$, then equation (\ref{eq:gen}) has no integer solutions for $d\geq 5$. 
\end{proof}

\begin{Cor}
If $\sF$ is a codimension one distribution with locally free tangent sheaf. If $d \notin \{1,2,12\}$, then $C$ can not be a nonsingular elliptic curve.
\end{Cor}
\begin{proof}
If $\chi(\OC)=0$, then equation (\ref{euler}) implies $d^3+2d^2+2d=(3d-2)\deg(C)$, which has integer solutions if and only if $d \in \{1,2,12\}$.
\end{proof}

\begin{Cor}\label{sing_reta}
If $\sF$ is a codimension one distribution with locally free tangent sheaf, then $C$ is an irreducible plane curve if and only if $\deg(C)=1$ and $\deg(\sF)=0$.
\end{Cor}
\begin{proof}
If $C$ is a plane curve then $2\chi(C)=2-2p_a(C)=-\deg(C)(\deg(C)-3)$.
Then equation (\ref{euler})
$$-\deg(C)(\deg(C)-3)= d^3+2d^2+2d-(3d-2)\deg(C)$$
and the only integer solution is given by $\deg(C)=1$ and  $d=0$, for $d\geq 0.$
\end{proof}

%--------------------------------------------------------

\subsection{Sub-foliations of codimension one distributions}

Let 
$$ \mathscr{G}:\  0  \longrightarrow \O(-k) \stackrel{\phi'}{ \longrightarrow} T\p3 \longrightarrow
N_\sG \longrightarrow 0 . $$
be a sub-distribution of a codimension one distribution $\sF$ on $\p3$, with $k\ge-1$; since $\mathscr{G}$ has dimension 1, it is automatically integrable.

In this case, the relative normal sheaf is a torsion free sheaf, being a subsheaf of $N_\sG$; it has rank 1 and degree $c_1(T_\sF)-(-k)=k+2-d$, thus $N_{\sG/\sF}\simeq I_{Y/\P^3}(k+2-d)$, for some subscheme $Y\subset\p3$. From the sequence
\begin{equation} \label{y}
0  \longrightarrow \O(-k) \stackrel{\sigma}{\longrightarrow} T_\sF
\longrightarrow I_{Y/\P^3}(k+2-d) \longrightarrow 0
\end{equation} 
we can see that $Y$ is just the zero locus of $\sigma\in H^0(T_\sF(k))$; note that $Y$ might be empty, in which case $T_\sF$ splits as a sum of line bundles. Therefore, assuming that $T_\sF(k)$ does not split as a sum of line bundles implies that $Y$ is non-empty. In addition, $Y$ contains $\sing(T_\sF)$.

Since $T_\sF$ is reflexive, $Y$ is a Cohen--Macaulay, generically locally complete intersection subscheme of pure dimension 1, see \cite[first paragraph of the proof of Theorem 4.1]{H2}. In addition, the sequence in display (\ref{sqc relat normal}) simplifies to
\begin{equation} \label{rn1}
0 \to I_{Y/\P^3}(k+2-d) \to N_\sG \to I_{Z/\P^3}(d+2) \to 0 .
\end{equation}
Dualizing the sequence (\ref{rn1}) we obtain
\begin{align} 
0\to \O(-d-2) \to N_\sG^\vee \to \O(d-2-k) \to \omega_C(2-d) \to \label{rn1.1} \\
\to \inext^1(N_\sG,\O) \to \omega_Y(2+d-k) \to \inext^3(\mathscr{U},\O) \to 0, \nonumber
\end{align}
where $\omega_{Y}$ and $\omega_{C}$ are the dualizing sheaves of the curves $Y$ and $C$, respectively, and we used the following identifications:
$$ \inext^1(I_{Y/\P^3}(k+2-d),\O)\simeq \omega_{Y}(d-k+2),$$
$$ \inext^1(\IZ(d+2),\O) \simeq \inext^1(\IC(d+2),\O)\simeq \omega_{C}(2-d),~~{\rm and} $$
$$ \inext^2(\IZ(k+2),\O) \simeq \inext^3(\mathscr{U},\O). $$
Breaking the long exact sequence in display \eqref{rn1.1}, we obtain the following short exact sequences 
\begin{equation}\label{v1}
0 \to \O(-d-2) \to  N_\sG^\vee \to \IC(d-2-k) \to 0
\end{equation}
\begin{equation}\label{v2}
0 \to \OC (d-2-k) \to \omega_{C}(2-d) \to \mathscr{V} \to 0,
\end{equation}
where $\mathscr{V}$ is a 0-dimensional sheaf whose support is contained in $C$; furthermore,
\begin{equation}\label{v3}
0 \to \mathscr{V} \to \inext^1(N_\sG,\O) \to L \to 0
\end{equation}
where $L$ is a torsion free sheaf on $Y$ defined by the sequence
\begin{equation}\label{v4}
0\to L \to \omega_Y(2+d-k) \to \mathscr{U} \to 0.
\end{equation}
Note that $L$ must be a sheaf of pure dimension 1 supported on $Y$, since it is a subsheaf of $\omega_Y$ with 0-dimensional quotient.
We therefore conclude that 
\begin{equation} \label{incl}
\sing(T_\sF) \subset (\sigma=0) \subset \sing(N_\sG)=\sing(\sG).
\end{equation}
Furthermore, the 0-dimensional component of $\sing(N_\sG)$ is precisely the support of the sheaf $\mathscr{V}$, which in contained in $C$, while $Y$ is precisely the pure 1-dimensional component of $\sing(N_\sG)$.

In other words, we have proved the following claim.

\begin{Lemma}\label{sing sub-dist}
Let $\sF$ be a codimension one distribution in $\p3$, and let $\sG$ be a sub-foliation induced by a nonzero section $\sigma\in H^0(T_\sF(k))$. 
\begin{itemize}
\item If zero locus of $\sigma$ is non-empty, then the 1-dimensional component of $\sing(\sG)$ coincides with $(\sigma=0)$, while the 0-dimensional component of $\sing(\sG)$ is contained in the 1-dimensional component of $\sing(\sF)$.
\item If $\sigma$ does not vanish, which can only occur when $T_\sF$ splits as a sum of line bundles, then $N_\sG$ is a reflexive sheaf whose singular locus is contained in the 1 dimensional component of $\sing(\sF)$. 
\end{itemize}
\end{Lemma}

\begin{Obs}\label{JEK}\rm
Let $\sG$ be a foliation of codimension 2 on $\p3$; recall that a curve $D\subset\p3$ is said to be \emph{invariant under $\sG$} if there is a nontrivial morphism $\omega_D\to T_\sG|_D$. Jouanolou \cite[Cor. 4.2.7 ]{J} and Esteves--Kleiman \cite[Prop. 3.3]{EK} have shown that if $D$ is invariant under $\sG$, then $\sing(\sG)\cap D\ne\emptyset$. It follows that the induced section $\nu\in H^0(D, (\omega_D)^{\vee}\otimes T_\sG|D)$  has nonempty zero locus $(\nu=0)$ of dimension $0$.
\end{Obs}

%--------------------------------------

\subsection{Connectedness of the singular locus}
Motivated by a conjecture due to Cerveau \cite{Ce} regarding the connectedness of the 1-dimensional component of the singular set a codimension one foliation on $\p3$, we consider the analogous problem for codimension one distributions on $\p3$. In this section, we present a homological criterion for connectedness, and explore the relation between connectedness and integrability.

\begin{Theorem}\label{con}
Let $\sF$ be a codimension one distribution on $\mathbb{P}^3$ of degree $d>0$  with singular scheme $Z$. If $h^2(T_\sF(-2-d))=0$, then $Z$ is connected and of pure dimension 1, so that $T_\sF$ is locally free. Conversely, for $d\ne2$, if $Z$ is connected, then $T_\sF$ is locally free and $h^2(T_\sF(-2-d))=0$.
\end{Theorem}

\begin{proof}
Twisting the exact sequence (\ref{eq:Dist p3}) by $\op3(-2-d)$ and passing to cohomology we obtain, using Bott's formula:
$$ 0\to H^1(I_{Z/\P^3}) \to  H^2(T_\sF(-2-d)) \to H^2(T\p3(-2-d)). $$
If $h^2(T_\sF(-2-d))=0$, then $h^1(I_{Z/\P^3})=0$. It follows from the sequence
$$ 0 \to I_{Z/\P^3} \to \op3 \to \OZ \to 0 $$
that 
$$ H^0(\O)\rightarrow H^0(\OZ)\rightarrow 0, $$
hence $h^0(\OZ)=1$. From the sequence (\ref{OZ->OC}), we get
$$ 0 \to H^0(\mathscr{U}) \to H^0(\OZ) \to H^0(\OC) \to 0 $$ 
Thus either $h^0(\OC)=1$, and $\mathscr{U}=0$ and $C$ is connected, or
$\mathrm{length}(\mathscr{U})=1$ and $C$ is empty. This second possibility leads to a contradiction: by equation (\ref{eq:ising}) $\mathrm{length}(\mathscr{U})$ is either equal $0$ or $\ge5$. 
It follows that $Z=C$ must be connected and of pure dimension 1, and thus, by Lemma \ref{loc free tg}, $T_\sF$ is locally free. 

Conversely, assume that $Z$ is connected; by the argument of the previous paragraph, $Z$ cannot consist of a single point. Thus $Z$ must be of pure dimension 1, and Lemma \ref{loc free tg} implies that $T_\sF$ is locally free. It also follows that $h^1(I_{Z/\P^3})=0$; since $h^2(T\p3(-2-d))=0$ for $d\ne2$, we conclude that $h^2(T_\sF(-2-d))=0$, as desired.
\end{proof}

When the tangent sheaf $T_\sF$ splits as a sum of line bundles, Giraldo and Pan-Collantes proved that the singular locus $Z$ is arithmetically Cohen-Macaulay, hence, in particular, $Z$ is connected; see \cite[Theorem 3.3]{GP}. We now provide a new proof  of the connectedness, as a particular case of Theorem \ref{con}.

\begin{Cor}
If $\sF$ is a codimension one distribution on $\P^3$ whose tangent sheaf splits as a sum of line bundles, then its singular scheme is connected.
\end{Cor}

\begin{proof}
Assuming that $T_\sF=\O(1-d_1)\oplus \O(1-d_2)$ with $d_1,d_2\ge0$, then clearly  $h^2(T_\sF(-2-d))=0$, where $d=d_1+d_2$
is the degree of $\sF$.
\end{proof}

\begin{Cor}
Let $\sF$ be a codimension one distribution on $\p3$ with locally free tangent sheaf. 
If $T_\sF^{\vee}$ is ample, then its singular scheme is connected.
\end{Cor}
\begin{proof}
We have, by Serre duality, 
$$ H^2(T_\sF(-2-d)) \simeq H^1(T_\sF^{\vee}(2+d)\otimes K_{\P^3})=H^1(T_\sF^{\vee}(d-2)\otimes \O(4)\otimes K_{\P^3}). $$
Observe that $T_\sF^{\vee}(d-2)\otimes \O(4)\otimes K_{\P^3}=T_\sF^{\vee}\otimes \det(T_\sF^{\vee})\otimes \O(4)\otimes K_{\P^3}$; since $T_\sF^\vee$ and $ \O(4)$ are ample, then, by Griffiths Vanishing Theorem \cite[section 7.3.A, pg 335 ]{La}, we get
$$
h^2(T_\sF(-2-d)) = h^1(T_\sF^{\vee}\otimes \det(T_\sF^{\vee})\otimes \O(4)\otimes K_{\P^3}) =0 .
$$
The result follows from Theorem \ref{con}.
\end{proof}

We now examine the relation between connectedness and integrability. Let $\o$ be a germ of an integrable 1-form in $\C^3$, singular at $0\in\C^3$. There are three possibilities.

\begin{itemize}
\item[(1)] Isolated singularities.  If $0$ is an isolated singularity of $\o$, then it  follows from Malgrange's Theorem \cite{Ma} that 
there are germs of holomorphic functions $f\in \mathcal{O}_0$ and
$g\in \mathcal{O}_0^{\ast}$, such that $\o=gdf$. In this case,  the tangent sheaf $\ker(\o)$ is not  locally free at $0$.

\item[(2)] Kupka type singularities. Define $K(\o)=\{\omega(0)=0,\,\, d\omega(0)\neq0\}$. It follows from Kupka's Theorem \cite{K} that $K(\o)$ is a smooth curve and, locally, the form $\o$ can  be written in two variables. In this case,  the tangent sheaf $\ker(\o)$ is locally free at $0$.

\item[(3)] Simple singularities. $\o(0)=0$ and $d\o$ has an isolated singularity at $0$. These singularities are contained in the closure $\overline{K(\o)}$. The tangent sheaf is locally free at $0$. See   \cite{CCGL} or \cite[Theorem 4.1]{CCF}.
\end{itemize}

%We define the non-Kupka set of $\o$ by $NK(\o)=\sing(\o)\setminus K(\o)$. This type of singularities was studied in \cite{CCF}.

Finally, we recall a particular case of the Baum--Bott residues formula for a codimension one foliation $\sF$ on a complex compact manifold $X$. Set $X^0=X\setminus \sing(\sF)$ and take $p_{0}\in X^0$. Then, in a neighborhood $U_{\alpha}$ of $p_{0}$, the foliation $\sF$ is induced by a holomorphic  $1$--form $\o_{\a}$ and there exists a differentiable  $1$--form $\theta_{\a}$ such that
\[
d\o_\a=\theta_\a \wedge \o_\a
\]
Denote by $C$ the codimension 2 component of $\sing(\sF)$. Let $C_j$ be an irreducible component of $C$. Take a generic point $p\in c_j$, that is, $p$ is a point where $C_j$ is smooth and disjoint from the other singular components. Pick $B_{p}$ a sufficiently small ball centered at $p$, so that $S(B_p)$ is a sub-ball of $B_p$ of codimension $2$. Then the De Rham class can be integrated over an oriented $3$-sphere $L_{p}\subset B^{*}_{p}$ positively linked with $S(B_{p})$:
\[
\BB(\sF, C_j)=\frac{1}{(2\pi i)^{2}}\int_{L_{p}}\theta\wedge d\theta.
\]
This complex number is the \textit{Baum--Bott residue of $\sF$ along $C_j$}. This formula is a particular case of the general Baum--Bott residue Theorem \cite{BB} reproved by Brunella and Perrone in \cite{BP} and \cite{CF} for foliations of higher codimension.  

If $\sF$ is a codimension one foliation on $\mathbb{P}^{3}$ of degree $d$, then the Baum--Bott Theorem yields the following equation 
\begin{equation}\label{Baum-Bott}
\sum_{C_j \subset C} \BB(\sF, C_j) \deg(C_j)=(d+2)^{2}.
\end{equation}
For more details see \cite[p. 2075]{CCF}. 
\begin{Theorem} \label{sm-conn}
Let $\sF$ be a codimension one distribution on $\mathbb{P}^3$, of degree $d > 0$. If $T_\sF$ is locally free and $\sing(\sF)$ is smooth and connected, then $\sF$ is not integrable. 
\end{Theorem}

\begin{proof}
Assume that $\sF$ is integrable. Consider the morphism $\varphi: T_\sF \longrightarrow T\mathbb{P}^3$ defining the distribution and the associated twisted $1$-form $\omega \in H^0(\Omega_{\P^3}^1(d+2))$. Since $T_\sF$ is locally free, the singular set of the distribution is 
$$ \sing(\sF) = \{ p \in \mathbb{P}^3 |\ \mathrm{rk} ( \varphi_p) \leq 1 \} = 
\{ p \in \mathbb{P}^3 |\ \omega(p)=0 \}. $$
In fact,  by taking wedge product we get the morphism  $\wedge^2\varphi : \det(T_\sF) \simeq \O(-d+2) \longrightarrow \wedge^2T\P^3  $. Since $\wedge^2T\P^3 \simeq \Omega_{\P^3}^1(4)$ we have that   $\wedge^2\varphi : \O(-d+2) \longrightarrow \Omega_{\P^3}^1(4)$ which induces the twisted $1$-form $\omega:  \O \longrightarrow \Omega_{\P^3}^1(d+2)$. 
Consider the dual morphism $\omega^{\vee}:T\P^3(-d-2) \longrightarrow \O$. Then, $\sing(\sF)$ is the subscheme of  $\P^3$ whose ideal is precisely the sheaf $\img(\omega^{\vee})$, and, as a set, it coincides with the zeros of $\omega$ and with the degeneracy locus of the map $\varphi : T_\sF \longrightarrow T\mathbb{P}^3$, see \cite[section 2.2]{CJV}. 

Since $Z$ is connected and smooth we  can consider a small neighborhood $U$ of $Z$  with  an analytic local coordinates system $(x_1,x_2,x_3)$ such  that $Z\cap U=\{x_1=x_2=0\}$. Now, take in $U$ a local frame $\{u, v\}$ of $T_\sF|_{U}$.  Then $\varphi(u) =\mathbf{X}$ and
$\varphi(v) =\mathbf{Y}$ are holomorphic vector fields such that
$\omega = \imath_{\mathbf{X}} \imath_{\mathbf{Y}} d x_1 \wedge d x_2 \wedge d x_3$, up to a unit in $\mathcal{O}_U$ , where $\imath_{\mathbf{X}}$ denotes the inner product of the vector field $\mathbf{X}$ with differential forms. 

Since  $\sing(\sF)$ is smooth, then
$$
D_0 ( \varphi) := \{ p \in \mathbb{P}^3 | \  \mathrm{rk} ( \mathrm{Im}(\varphi))_p = 0 \}=\emptyset . 
$$ 
That is, for all $p \in Z$ we have $\mathrm{rk} ( \varphi)_p = 1$. Therefore,  we can assume that $\mathbf{X}$ is a germ of a tangent vector field at $p\in U$ that does not vanish at $p$.  We can suppose that 
in the coordinate system  $\zeta = ( x_1, x_2, x_3)\in U$ that 
the vector field  is given by $\mathbf{X}= \partial / \partial x_3$ and that the foliation is induced by the $1$-form 
$\omega = A
( \zeta) d x_1 + B ( \zeta) d x_2 + C ( \zeta) d x_3$. 
Then,    we have that 
\begin{eqnarray*}
  \omega = e^{\alpha x_3} [ a ( x_1, x_2)d x_1 +   b( x_1, x_2) d x_2] ;&  \alpha \in \mathbb{C},  & 
\end{eqnarray*}
where $a ( x_1, x_2)$ and $b( x_1, x_2) $ are analytic functions. In fact,   since $\imath_{\mathbf{X}}
\omega \equiv 0,$ we have $C \equiv 0.$ 
Now, by integrability condition we have
$$ 0 \equiv  \omega \wedge d \omega = \left( A ( \zeta)
 \partial_{x_3} B  - B ( \zeta)  \partial_{x_3} A
\right) d x_1 \wedge d x_2 \wedge d x_3, $$
where $\partial_{x_3} $ denotes the derivation with respect to $x_3$. 
Then
$$
 A ( \zeta)
 \partial_{x_3} B  - B ( \zeta)  \partial_{x_3} A  \equiv 0. 
$$
This implies
\begin{eqnarray*}
A ( \zeta)
 \partial_{x_3} B  - B ( \zeta)  \partial_{x_3} A   = 0 \Longrightarrow A (
  \zeta) = e^{\alpha x_3} a ( x_1, x_2, 0), \hspace{1em} \mathrm{and} \hspace{1em} B (
  \zeta) = e^{\alpha x_3} b ( x_1, x_2, 0),  
  \alpha \in \mathbb{C}. &  & 
\end{eqnarray*}
Let $\omega =  j^1\omega+ j^2\omega +\cdots$ be the Taylor series of $\omega$ in a neighborhood of $p$.
We observe that $j^1 \omega \neq 0$: indeed, since $\{a=b=0\}=Z\cap U=\{x_1=x_2=0\}$, the ideals $I(a,b)$ and $I(x_1,x_2)$ must be equal, which implies that there exists an invertible matrix $G\in GL(2,\mathcal{O}_U)$  such that 
\[
\begin{pmatrix}
x _1\\ x_2
\end{pmatrix}=
\begin{pmatrix}
g_{11} & g_{12} \\ g_{21} & g_{22}
\end{pmatrix}
\cdot 
\begin{pmatrix}
a \\b
\end{pmatrix}\quad G=\begin{pmatrix}
g_{11} & g_{12} \\ g_{21} & g_{22}
\end{pmatrix},\quad g_{ij}\in \mathcal{O}_U
\]
Now, we take the Taylor series of $g_{ij},a,b$, for instance $a=a_1+a_2+\cdots + a_k+\cdots$, where $a_k$ has degree $k$.  Observe that $j^1\omega= a_1dx_1+b_1dx_2$. Thus, 
\[
\begin{pmatrix}
x _1\\ x_2
\end{pmatrix}=
(G_0+\cdots)\begin{pmatrix}
a_1+a_2+\cdots \\
b_1+b_2+\cdots
\end{pmatrix}
\]
where $G(0)=G_0 \in GL(2,\mathbb{C})$. We can rewrite this last equation as follows 
\[
\begin{pmatrix}
x _1\\ x_2
\end{pmatrix}=G_0\cdot \begin{pmatrix}
a_1=a_x(0) x_1+ a_y(0) x_2 \\ b_1=b_x(0) x_1 + b_y(0) x_2
\end{pmatrix}+\mbox{higher degrees terms},\quad 
\begin{pmatrix}
a_{x_1}(0) & a_{x_2}(0) \\ b_{x_1}(0) & b_{x_2}(0)
\end{pmatrix}\in GL(2,\mathbb{C}),
\]
where $a_{x_1},a_{x_2},b_{x_1},b_{x_2}$ denote the respective partial derivatives. Comparing the linear terms, we conclude that 
\[
\begin{pmatrix}
x _1\\ x_2
\end{pmatrix}=G_0\cdot \begin{pmatrix}
a_1=a_x(0) x_1+ a_y(0) x_2 \\ b_1=b_x(0) x_1 + b_y(0) x_2
\end{pmatrix}.
\]
In particular, either $a_1\neq 0$  or $b_1\neq 0$, hence $j^1\omega= a_1dx_1+b_1dx_2 \neq 0$.    

If $K(\omega) \cap Z = \emptyset$, then by \cite{CCF} the Baum--Bott residue $\BB(\sF,Z)$ along $Z$ vanishes, providing a contradiction with the Baum--Bott formula \ref{Baum-Bott}, which yields
\begin{eqnarray*}
\BB(\sF,  Z) \deg ( Z) = ( d + 2)^2. &  & 
\end{eqnarray*}

Now, suppose that $K(\omega)\cap Z \neq\emptyset$. Since $Z$ is connected and smooth, we conclude that $K(\omega)=Z$. 
Indeed, since $Z$ is connected it is enough to show that $K(\sF)\cap Z$ is an open and closed subset of $Z$. It is clearly an open set $\{p| \omega(p)=0,\quad d\omega(p)\neq0\}$ by definition. 

It remains for us to show that $K(\sF)\cap Z\subseteq Z$ is closed. Let $p\in Z$ be a point in the closure $\overline{K(\sF)\cap Z}$; in a small neighborhood $V$ of $p$, we can take a disc $\Sigma$ transversal to $\sF$, and we have  a foliation  $\sF|_{V}\cap \Sigma$  on $\Sigma$ such that $\{p\}=\sing(\sF|_{V}\cap \Sigma)$. 

Similarly to what was done above, we consider a vector field $\mathbf{W}$ tangent to the foliation defined in $V$ such that $\mathbf{W}(p)\neq0$.  Consider $\Phi_{\epsilon}$ the flow of $\mathbf{W}$ such that in the time $\epsilon$ we have $\Phi_{\epsilon}(p):=q\in K(\sF)\cap Z$. Fixed $\epsilon$,  we have that $\Phi_{\epsilon}: (V,p)\to  (\Phi_{\epsilon}(V),q) $ is the germ of a biholomorphism, so that $\sF\cap \Sigma_{\epsilon}(V)$ is biholomorphic to $\sF|_{V}$. Since the Kupka property is preserved by biholomorphisms, we conclude that $\sF|_{V}$ is Kupka at $p$. This shows that $ K(\sF)\cap Z$ is closed.

It follows from \cite{Br1,C1,C2} that $K(\omega) = Z$ is a complete intersection curve, given by $\{ f=g=0 \}$, where $\deg(f)+\deg(g) = d + 2$; in addition, since $\deg(\sF)>0$, a result due to Cerveau and Lins Neto \cite{CL2} implies that $\sF$ has a  rational  first integral of the type $f^p / g^q$. However, such foliations always have isolated singularities, thus contradicting our hypothesis of $T_\sF$ being locally free; see \cite{CSV} or Theorem \ref{racionais} below.
\end{proof}

% --------------------------

\subsection{Martinet surface}

The \emph{O'Neill tensor} of a codimension one distribution $\sF$ on $X=\p3$ is defined by the morphism \cite{JZ}:
$$
\begin{array}{cccc}
& \bigwedge^2 T_\sF& \longrightarrow & N_{\sF}\simeq I_Z(d+2) \\
\\
 & u \wedge v & \longmapsto &\pi[\phi(u),\phi(v)],
\end{array}
$$
This morphism is trivial precisely when $\sF$ is integrable. Otherwise, since $\bigwedge^2 T_\sF= \O(2-d)$, we get a nonzero element of $H^0(\IZ(2d))$. This means that $\sing(\sF)$ is contained in a surface $V$ of degree $2d$, called the \emph{Martinet surface} of $\sF$, which is given by the zero set of the non-trivial section  
$$
\omega \wedge d\omega \in H^0(\P^3, \Omega^3_{\P^3}(2d+4)) =H^0(\P^3, \O(2d)). 
$$
Observe that the Martinet surface of a non-integrable distribution of degree 0 is empty, since in this case  $\omega \wedge d\omega$ is a non-trivial section of $H^0(\P^3,\O)$.  

Our next result poses certain restrictions on the Martinet surface of non-integrable distributions on $\p3$. % of higher degree.

\begin{Prop}
Let $\sF$ be a non-integrable locally free distribution on $\P^3$ of degree $d>0$. 
Every smooth hypersurface of degree $2d$ containing $\sing(\sF)$ has Picard rank larger than 1.
\end{Prop}

In particular, the Martinet surface of a locally free non-integrable distribution either is singular or has Picard rank at least 2. It would be interesting to know what is the smallest Picard rank of the Martinet surface of a non-integrable distribution of degree $d$ and whether this surface is smooth. 

\begin{proof}
Let $V\subset\p3$ be a smooth hypersurface of degree $2d$ such that $\Pic(V)=\mathbb{Z}\cdot \mathcal{O}_V(1)$, where $\mathcal{O}_V(1)=\mathcal{O}_{\mathbb{P}^3}(1)|_V$. If $\sing(\sF)=C\subset V$, then $\mathcal{O}_V(C)=\mathcal{O}_V(r)$, for some $r$. By the adjunction formula, we have
$$ 2p_a(C)-2= c_1(\mathcal{O}_V(C))\cdot(c_1(\mathcal{O}_V(C)) +c_1(K_V)) = 2dr(r+2d-4), $$
since $K_V=\mathcal{O}_V(2d-4)$. This previous formula together with the fact that $T_\sF$ is locally free and Theorem \ref{P:SLocus} implies that
$$ d^3+2d^2+2d-(3d-2)(2dr)+ 2dr(r+2d-4)= 0, $$
since $\deg(C)=2dr$. However, this equation has no integer solutions for $d>0$.
\end{proof}
 
%%%%%%%%%%%%%%%%%%%%%%%%%%%%%%%%%%%%%%%%%%%%%%%%%%%%%%%%%%%%%%%%%%%%%%%%%%%%%%%%%%%%%%
%%%%%%%%%%%%%%%%%%%%%%%%%%%%%%%%%%%%%%%%%%%%%%%%%%%%%%%%%%%%%%%%%%%%%%%%%%%%%%%%%%%%%%

\section{Cohomology of distributions}\label{ss:cohomology}

We start by listing a few elementary facts about the cohomology of the tangent sheaf $T_\sF$ of a codimension one distribution $\sF$ as in display (\ref{eq:Dist p3}).

\begin{Lemma}\label{L:CH}
The tangent sheaf $T_\sF$ of a codimension one distribution $\sF$ of degree $d$ satisfies:
\begin{enumerate}
\item[(i)] $T_\sF^{\vee} =T_\sF(d-2)$;
\item[(ii)] $h^0(T_\sF(p))=0\mbox{ for } p\leq -2$;
\item[(iii)] $h^1(T_\sF(p))=0$ for $p\leq -d-2$; 
\item[(iv)] $h^3(T_\sF(p))=0$ for $p\geq d-4$.
\end{enumerate}
\end{Lemma}

\begin{proof}
Since $T_\sF$ is a rank 2 reflexive sheaf, the first claim is just $T_\sF^{\vee}=T_\sF\otimes (\det T_\sF)^{\vee}$ and $(\det T_\sF)^{\vee}=\O(-c_1(T_\sF))$ \cite[Prop 1.10]{H2}. The second item follows easily from (\ref{eq:Dist}), since $h^0(T\mathbb{P}^3(-2))=0$. For the third one, we have the exact sequence in cohomology
\[
H^0(I_{Z/\P^3}(d+p+2))\rightarrow H^1(T_\sF(p))\rightarrow H^1(T\mathbb{P}^3(p)).
\]
The term on the left vanishes for $p+d+2\leq 0$; on the other hand, the term of the right vanishes for all $p$. Finally, the last claim is just Serre duality, see \cite[Theorem 2.5]{H2}. 
\end{proof}

\begin{Obs} \label{min twist} \rm
Let $E$ be a reflexive sheaf rank two on $\p3$, and consider 
$$ r:= \min \{t\in\Z ~|~ h^0(E(t))\ne 0 \} .$$
Given a nontrivial section $\sigma\in H^0(E(r))$, we note that the cokernel $K$ of the associated morphism $\op3(-r)\to E$ is a torsion free sheaf.

Indeed, assume that $K$ is not torsion free, and let $T$ be its maximal torsion subsheaf. Let $L$ be the kernel of the composed epimorphism $E \twoheadrightarrow K \twoheadrightarrow K/T$. Since $E$ is reflexive and $K/T$ is torsion free, $L$ must be reflexive \cite[Lemma II.1.1.16 and Lemma II.1.1.12]{OSS}; but $L$ has rank 1, so $L=\op3(-p)$ for some $p$ \cite[Lemma II.1.1.15]{OSS}. 

\begin{equation}
\xymatrix{
        &                              &                     & 0 \ar[d]          & \\
        &                              &                     & T \ar[d]          & \\
0\ar[r] & \op3(-r) \ar@{..>}[d] \ar[r] & E \ar@{=}[d] \ar[r] & K \ar[r] \ar[d]   & 0 \\
0\ar[r] & L \ar[r]                     & E \ar[r]            & K/T \ar[r] \ar[d] & 0 \\
        &                              &                     & 0                 & \\
}
\end{equation}

From the exact sequence
$$ 0 \to \op3(-r) \to L \to T \to 0, $$
we must have $-p>-r$, contracting the fact the $r$ is the minimal twist for which $E$ has a section. 
\end{Obs}

\begin{Lemma}\label{coho_non_split}
Let $T_\sF$ be the tangent sheaf of a codimension one distribution of degree $d$. If $T_\sF$ does not split as a sum of line bundles, then $h^0(T_\sF(-1))=0$.
\end{Lemma}

\begin{proof}
Given a nontrivial section $\sigma\in H^0(T_\sF(-1))$, note from Remark \ref{min twist} that the cokernel of the associated morphism $\O(1)\to T_\sF$ is torsion free because $h^0(T_\sF(p))=0$ for every $p<-1$; in addition, its zero locus is nonempty, since $T_\sF$ does not split as a sum of line bundles. We obtain a sub-foliation
$$ \mathscr{G}:\  0  \longrightarrow \O(1) \longrightarrow T\p3 \longrightarrow N_\sG \longrightarrow 0 , $$
whose singular locus consists of a single point, see Remark \ref{Obs-subzero}. However, $\sing(N_\sG)$ should contain, by the inclusion in display (\ref{incl}), the zero locus of $\sigma$, which is 1-dimensional. It thus follows that $h^0(T_\sF(-1))=0$.
\end{proof}

Now let $c:=c_2(T_\sF)$ and $\ell=c_3(T_\sF)$ denote the second and third Chern classes of the tangent sheaf $T_\sF$. The Hilbert polynomial of $T_\sF$, defined by $P_\sF(t)=\chi(T_\sF(t))$, is given by
\begin{equation}\label{eq:HilbP}
P_\sF(t)=\frac{1}{3}(t+3)(t+2)(t+1)+\frac{1}{2}(t+2)(t+1)(2-d)-(t+2)c+\frac{1}{2}(\ell +(d-2)c)
\end{equation}

\begin{Obs}\label{R:CC}\rm
Note that $\ell\geq0$, and $\ell=0$ if and only if $T_\sF$ is locally free \cite[Proposition 2.6]{H2}. Moreover, $dc\equiv \ell (mod2)$ \cite[Cor. 2.4]{H2}, and if $T_\sF$ is locally free, then $c$ must be even whenever $d$ is odd.
\end{Obs}

Now, we focus on distributions with locally free tangent sheaf.

\begin{Lemma}\label{L:Csdg}
Let $T_\sF$ be the tangent sheaf of a codimension one distribution of degree $d$ with reduced singular scheme. If $T_\sF$ is locally free, then:
\begin{enumerate}
\item[(i)] For $d=0,\, h^1(T_\sF(p))=0$ for $p\neq-2$ and $h^1(T_\sF(-2))\leq 1$

\item[(ii)] For $d=1,\, h^1(T_\sF(p))=0$ for $p\leq -3$ and $p\geq0$, with $h^1(T_\sF(-1))\leq 1$.

\item[(iii)] For $d\geq 2,\, h^1(T_\sF(p))=0$ for $p\leq -d-2$ and $p\geq 2d-3$.
\end{enumerate}
\end{Lemma}
\begin{proof}
If $T_\sF$ is locally free, then the singular scheme $Z$ has pure dimension 1; if, in addition, $Z$ is reduced, then $h^0(\OZ(k))=0$ for $k\leq -1$, hence $h^1(I_{Z/\P^3}(k))=0$ in the same range. Therefore, in the cohomology sequence
\[
H^1(I_{Z/\P^3}(p+d+2))\rightarrow H^2(T_\sF(p))\rightarrow H^2(T\mathbb{P}^3(p))
\]
the left term vanishes for $p\leq -d-3$, while the right vanishes except for $p=-4$, and $h^1(T\mathbb{P}^3(-4))=1$. It follows that
\begin{enumerate}
\item[(i)] for $d=0,\, h^2(T_\sF(p))=0$ for $p=-3$ and $p\leq -5$, with $h^2(T_\sF(-4))\leq1$.

\item[(ii)] for $d=1,\, h^2(T_\sF(p))=0$ for $p\leq-5$, with $h^2(T_\sF(-4))\leq1$
\item[(iii)] for $d\geq2,\, h^2(T_\sF(p))=0$ for $p\leq -d-3$.
\end{enumerate}

The lemma follows easily from this observation, item (i) of Lemma 2.1, and Serre duality: $h^1(T_\sF(p))=h^2(T_\sF(d-p-6))$.
\end{proof}

The next result will be used in the classification of distributions with locally free tangent sheaf of small degree.

\begin{Lemma}\label{L:C2}
If $T_\sF$ is the tangent sheaf of a codimension one distribution of degree $d\geq2$, then $c_2(T_\sF)\leq d^2+2$. If, in addition, $T_\sF$ is locally free and its singular scheme is reduced, then
\[
c_2(T_\sF)\leq \dfrac{(2d-1)(2d-2)(d/3+2)}{3d}
\]
\end{Lemma}
\begin{proof}
The first inequality comes from Theorem \ref{P:SLocus}, using the fact that $\deg(C)\ge1$ when $d\geq2$.

From Lemma \ref{L:CH} item (iv) and Lemma \ref{L:Csdg} item (ii) plus Serre duality, we have that $h^q(T_\sF(p))=0,\, q=1,2,3$ and for $p\geq 2d-3$. It follows that
\[ \chi(T_\sF(2d-3))=h^0(T_\sF(2d-3))\geq0. \]
Applying this information into formula (\ref{eq:HilbP}), we obtain the inequality
\[ \frac{1}{3}2d(2d-1)(2d-2)+\frac{1}{2}(2d-1)(2d-2)(2-d)-\frac{3d}{2}c\geq0 \]
from which we obtain the desired estimate.
\end{proof}

An effective lower bound for the second Chern class of the tangent sheaf of a codimension one distribution can be obtained as follows. Consider a generic plane $P\hookrightarrow \P^3$; a codimension one distribution $\sF$ on $\p3$ induces a dimension one foliation on $P\simeq\p2$ of the same degree; the singular scheme of this distribution comes from the intersection of $P$ with the singular scheme of $\sF$, plus some tangency scheme. Using the bounds for foliations found in \cite[Corollary 4.8]{Soares}, we have that 
\begin{equation}\label{eq:BdegZ}
\deg(C)\leq d^2+d+1 .
\end{equation}
Using the identities of Theorem \ref{P:SLocus}, we obtain
\begin{equation}\label{lb c2}
c_2(T_\sF) \ge 1-d .
\end{equation}

The lower bound in inequality (\ref{lb c2}) is realized in every degree by foliations admitting a sub-foliation of degree 0. Indeed, Let $T_\sF$ be the tangent sheaf of a codimension one distribution of degree $d$; this admits a sub-foliation of degree 0 if and only if $h^0(T_\sF(-1))\ne0$. By Lemma \ref{coho_non_split}, it follows that $T_\sF$ must split as a sum of line bundles, so that $T_\sF=\O(1-a)\oplus\O(1-b)$ with $0\le a\le b$ and $d=a+b$. Since $h^0(T_\sF(-1))\ne0$, we must have $a=0$, thus in fact $T_\sF=\O(1)\oplus\O(1-d)$, hence $c_2(T_\sF) = 1-d$.

\begin{Obs}\label{Obs-subzero}\rm
We recall that a one-dimensional foliation $\G$ of degree 0 on $\p3$ is given by the short exact sequence
$$ \G ~:~ 0 \to \O(1)  \to  T\p3 \to N_\G \to O, $$
where $N_\sG$ is a reflexive sheaf  since the  singular set of $\G$ consists of a single point. Explicitly, such a foliation is defined by a linear projection  $\rho :\p3 \dashrightarrow \p2$, see \cite[Th\'eor\`eme 3.8]{DC}; without loss of generality, we can suppose that $\rho(z_0:z_1:z_2:z_3) =(z_0:z_1:z_2)$, so that its singular point is precisely $[0:0:0:1]$. Since the foliation $\G$ is tangent to $\rho$, we have that it is induced by the vector field $\frac{\partial}{\partial z_3}$.   
\end{Obs}

With the previous observation in mind, we remark that distributions $\sF$ of degree $d$ with $T_\sF=\O(1)\oplus\O(1-d)$ are given, in homogeneous coordinates, by a 1-form of the type
\[
\omega=A_0(z_0,z_1,z_2,z_3)dz_0+A_1(z_0,z_1,z_2,z_3)dz_1+A_2(z_0,z_1,z_2,z_3)dz_2,
\]
where $A_j,\, j=0,1,2$ are homogeneous of degree $d+1$ and $\sum_{i=0}^2 z_iA_i\equiv 0$.
In fact, it follows from the Remark \ref{Obs-subzero} that the sub-foliation given by the inclusion $\O(1) \hookrightarrow \O(1)\oplus\O(1-d)=T_\sF$ can be represented by the constant vector field $\frac{\partial}{\partial z_3}$. This implies that $\o(\frac{\partial}{\partial z_3})=A_3 = 0$.

Observe that $\sF$ is integrable, i.e. $\o\wedge d\o=0$, if and only if each $A_j$ depends only on the variables $z_0,z_1,z_2$; this case corresponds to foliations which are a linear pull--back of foliations of degree $d$ on the plane $\P^2$.

%%%%%%%%%%%%%%%%%%%%%%%%%%%%%%%%%%%%%%%%%%%%%%%%%%%%%%%%%%%%%%%%%%%%%%%%%%%%%%%%%%%%%%%%%%%%%%%
%%%%%%%%%%%%%%%%%%%%%%%%%%%%%%%%%%%%%%%%%%%%%%%%%%%%%%%%%%%%%%%%%%%%%%%%%%%%%%%%%%%%%%%%%%%%%%%

\section{Distributions induced by globally generated sheaves}
\label{ss:ex}

In this section we describe a simple way to construct new examples of codimension one distributions on $\p3$, based on the following result, proved in the Appendix.

\begin{Prop}\label{global}
Let $E$ be a globally generated rank 2 reflexive sheaf on $\P^3$.
Then $E(1-c_1(E))$ is a tangent sheaf of a codimension one distribution $\sF$ of degree $d=c_1(E)$, with
$c_2(T_\sF)=c_2(E)-c_1(E)+1$ and $c_3(T_\sF)=c_3(E)$. %Moreover, if $E$ is locally free, then the singular set of the induced distribution is a smooth curve. 
\end{Prop}

Since every coherent sheaf becomes globally generated after a sufficiently large twist, it follows that for every reflexive sheaf $E$ there exists $p\in\Z$ such that $E(1-c_1(E)-p)$ is the tangent sheaf of a codimension 1 distribution on $\P^3$. However, there are codimension one distributions on $\p3$ that do not arise as in Proposition \ref{global}.

\begin{Ex} \rm \label{ex nc bdl}
Recall that a \emph{null-correlation bundle} is a locally free sheaf $N$ of rank 2 defined by the following exact sequence \cite[Lemma 4.3.2, p. 362]{OSS}
\begin{equation}\label{nc-bdl}
0 \to \O(-1) \to \Omega^1_{\p3}(1) \to N \to 0 ; %\stackrel{\sigma}{\to}
\end{equation}
%where $\sigma\in H^0(\Omega^1_{\p3}(2))$ is a non-vanishing section;
note that $c_1(N)=0$ and $c_2(N)=1$. Since $\Omega^1_{\p3}(1+k)$ is globally generated for $k\ge1$, then so is $N(k)$. Proposition \ref{global} provides us with codimension one distributions of degree $2k$ whose tangent sheaf is $N(1-2k)$.

Note that $N^\vee\simeq N$, since $N$ is a rank 2 locally free sheaf with trivial determinant. Therefore, dualizing sequence (\ref{nc-bdl}), and twisting the result by $\O(1)$ we obtain the following codimension one distribution of degree 0
\begin{equation}\label{nc-bdl dual}
0 \to N(1) \to T\p3 \to \O(2) \to 0 . %\stackrel{\sigma^\vee}{\to}
\end{equation}
Since $N$ is not globally generated, this distribution does not arise as in Proposition \ref{global}.

Moreover, we have that there is an epimorphism
$$
\O(k-1)^{\oplus 4} \rightarrow N(k)\rightarrow 0.
$$
for each $k\ge1$. Since $\O(k-1)^{\oplus 4}$ is ample for $k\geq 2$, then $N(k)$ is also ample in this range. Therefore, $T_\sF= N(1-2k)$ is a tangent sheaf of a codimension one distribution of degree $2k$ with $T_\sF^{\vee}=N(2k-1)$ ample, for all $k\geq 1$.
\end{Ex}

Another interesting property of the null correlation bundle as the tangent sheaf of a codimension one distribution is provided by the classification of indecomposable rank 2 globally generated locally free sheaves given in \cite[Theorem 1.1]{ChE}.

\begin{Cor}
Let $\sF$ a codimension one distribution of degree $\leq 3$. If its tangent sheaf $T_\sF$ is a globally generated locally free sheaf which does not split as a sum of line bundles, then $\deg(\sF)=2$ and  $T_\sF$ is a null correlation bundle twisted by $\O(-1)$.
\end{Cor}

%%%%%%%%%%%%%%%%%%%%%%%%%%%%%%%%%%%%%%%%%%%%%%%%%%%%%%%%%%%%%%%%%%%%%%%%%%%%%%%%%%%%%%%%%%%%%
%%%%%%%%%%%%%%%%%%%%%%%%%%%%%%%%%%%%%%%%%%%%%%%%%%%%%%%%%%%%%%%%%%%%%%%%%%%%%%%%%%%%%%%%%%%%%

\section{Stability of the tangent sheaves of distributions with isolated singularities} \label{isol sing}

Let $L$ be a very ample line bundle on a variety $X$ of dimension $n$; the slope of a torsion free sheaf $E$ on $X$ is given by
$$ \mu_L(E) := \frac{c_1(E)\cdot L^{n-1}}{\rk(E)} . $$
Recall that $E$ is said to be $\mu$-(semi)stable if every proper nontrivial subsheaf $F\subset E$ with $\rk(F)<\rk(E)$ and such that $E/F$ is torsion free satisfies $\mu_L(F)<\mu_L(E)$ ($\mu_L(F)\le\mu_L(E)$).

In this section, we study the stability of the tangent sheaf of codimension one distributions with only isolated singularities on $X=\p3$, setting $L=\O(1)$; in this case, the first Chern class of a sheaf will simply be regarded as an integer number, indicating the appropriate multiple of the ample generator of $\Pic(\p3)$. Recall that one has the following chain of strict implications \cite[Lemma 1.2.13]{HL}:
\begin{center}
$\mu$-stable $~~\Longrightarrow~~$ stable $~~\Longrightarrow~~$ semistable $~~\Longrightarrow~~$ $\mu$-semistable. 
\end{center}
However, a rank 2 reflexive sheaf is stable if and only if it is $\mu$-stable \cite[Corollary, page 176]{OSS}; by contrast, semistability and $\mu$-semistability are not equivalent.

% In this section, we study the stability of the tangent sheaf of codimension one distributions with only isolated singularities on a smooth weighted projective complete intersection Fano 3-fold $X$ with rank one Picard group. Such varieties have been classified by Iskovskikh \cite{I1,I2} and Mukai \cite{Mu}; they are: 
% \begin{enumerate}
% \item the projective space $\mathbb{P}^{3}$;
% \item a quadric  hypersurface $X_2\subset\mathbb{P}^{4}$;
% \item a cubic hypersurface $X_3\subset\mathbb{P}^{4}$;
% \item an intersection $X_{2,2}$ of two quadric hypersurfaces in $\mathbb{P}^{5}$;
% \item a hypersurface of degree $4$ in the weighted projective space $X_4\subset\mathbb{P}(1,1,1,1,2)$; 
% \item a hypersurface of degree $6$ in the weighted projective space $X_6\subset \mathbb{P}(1,1,1,2,3)$;
% \item an intersection of a quadratic cone and a hypersurface of degree $4$ in $\mathbb{P}(1,1,1,1,1,2)$;
% \item an intersection $X_{2,3}$ of a quadric and a cubic in $\mathbb{P}^5$;
% \item an intersection $X_{2,2,2}$ of three quadrics in $\mathbb{P}^6$.
% \end{enumerate}
%Let $\Pic(X)=\mathbb{Z}\cdot H$ with $H$ ample
%Recall that the \emph{index} of $X$ is the integer $i_X>0$ such that $-K_X=i_X\cdot H$. 

The goal is to prove the following result.

% Recall that if $E$ is a rank 2 reflexive sheaf, recall that $E$ is (semi)stable if and only if $h^0(E(k))=0$ ($h^0(E(k-1))=0$) where $k=-c_1(E)/2$ if $c_1(E)$ is even, and $k=-(c_1(E)+1)/2$ if $c_1(E)$ is odd, see \cite[Lemma 1.2.5, p. 165]{OSS}.

\begin{Theorem}\label{StGDist}
Let $\sF$ be a codimension one distribution on $\p3$. If $\sing(\sF)$ is a nonempty, 0-dimensional scheme, then $T_\sF$ is stable.
\end{Theorem}

% The proof makes use of a result due to Flenner, cf. \cite[Satz 8.11]{flenner81}, providing the following formulas for the cohomology of $\Omega_X^q$:
% \begin{enumerate}
% \item[(i)] $h^q(X,\Omega_{X}^q) =1 $ for $0\le q\le 3$, 
% \item[(ii)] $h^p\big(X,\Omega_{X}^q(t)\big) = 0$ in the following cases
% \begin{itemize}
% \item $0<p<3$, $p+q\neq 3$ and either $p\neq q$ or $t\neq 0$;
% \item  $p+q > 3$ and $t>q-p$;
% \item $p+q < 3$ and $t<q-p$.
% \end{itemize} 
% \end{enumerate}

\begin{proof}
Consider a codimension one distribution $\sF$ on $\p3$ of degree $d=2-c_1(T_\sF)$, given by the exact sequence
$$ 0 \to T_\sF \stackrel{\phi}{\longrightarrow} T\p3 \stackrel{\o}{\longrightarrow} \IZ(4-c_1(T_\sF)) \to 0, $$
where $\o\in H^0(\Omega_{\p3}^1(4-c_1(T_\sF))$ is the 1-form representing $\sF$; note that
$$ \IZ = \img\{ ~ \o:T\p3(-4+c_1(T_\sF))\to\O ~ \}. $$
Since $\dim Z=0$ and $Z\ne\emptyset$, the Koszul complex associated to $Z$
\begin{equation}\label{B2}
0 \to \bigwedge^3 \left(T\p3(-4+c_1(T_\sF))\right) \to
\bigwedge^{2}\left(T\p3(-4+c_1(T_\sF))\right) \to  
T\p3(-4+c_1(T_\sF)) \stackrel{\o}{\longrightarrow} \IZ \to 0,
\end{equation}
is exact. 

Every rank 1 subsheaf of a reflexive sheaf with torsion free quotient is of the form $\O(p)$ \cite[p.166, proof of Lemma 1.2.5]{OSS}, so it is enough to consider nonzero sections $\sigma\in H^0(T_\sF(-p))$. Twist the sequence \eqref{B2} by $\O(4-c_1(T_\sF)-p)$ to obtain
\begin{equation}\label{B22}
0 \to \O(-4+2c_1(T_\sF)-p) \to \Omega^1_{\p3}(c_1(T_\sF)-p) \to
T\p3(-p) \stackrel{\o}{\longrightarrow} \IZ(4-c_1(T_\sF)-p) \to 0 
\end{equation}
where we used that $\bigwedge^3T\p3 \simeq \O(4)$ and $\bigwedge^2T\p3 \simeq \Omega^1_{\p3}(4)$. Clearly, the kernel of the last morphism is just $T_\sF(-p)$, thus we obtain the sequence
$$ 0 \to \O(-4+2c_1(T_\sF)-p) \to \Omega^1_{\p3}(c_1(T_\sF)-p) \to
T_\sF(-p) \to 0. $$

Computing cohomology, we have
\begin{equation}\label{B222}
0 \to H^0(\O(-4+2c_1(T_\sF)-p)) \to H^0(\Omega^1_{\p3}(c_1(T_\sF)-p))
\to H^0(T_\sF(-p))\to 0;
\end{equation}
it follows that $H^0(\Omega^1_{\p3}(c_1(T_\sF)-p))\ne0$ thus $c_1(T_\sF)-p\ge2$, or, equivalently $\mu(T_\sF) - \mu(\O(p)) \ge 2 - \mu(T_\sF)$. However, $\mu(T_\sF)<\mu(T\p3)=4/3$, because $T\p3$ is stable; it follows that $\mu(T_\sF) - \mu(\O(p)) \ge 2/3>0$, as desired.

\end{proof}

We complete this section by emphasizing that the hypothesis that $\sF$ has only isolated singularities is essential. First, Proposition \ref{global} allows us to construct the following example of a codimension one distribution whose tangent sheaf is locally free (hence its singular scheme has pure dimension 1), and not $\mu$-semistable. Indeed, recall that a monad is a complex of locally free sheaves of the form
$$ A \stackrel{\alpha}{\longrightarrow} B \stackrel{\beta}{\longrightarrow} C $$
such that $\alpha$ is injective and $\beta$ is surjective; the sheaf $\ker\beta/\img\alpha$ is called the cohomology of the monad; see \cite[Section II.3]{OSS} for general facts on monads.

\begin{Ex}\label{exeNonsemi}\rm
Start by considering a monad of the form
$$ \O(-3) \stackrel{\alpha}{\longrightarrow}
\O(-2)\oplus\O(-2)\oplus\O(2)\oplus\O(2) \stackrel{\beta}{\longrightarrow} \O(3) $$
where the morphisms $\alpha$ and $\beta$ are given by
$$ \alpha = \left( \begin{array}{c} -x_3 \\ x_2 \\ -x_1^5 \\ x_0^5 \end{array} \right) ~~{\rm and}~~
\beta = \left( \begin{array}{cccc} x_0^5 ~ & ~ x_1^5 ~ & ~ x_2 ~ & ~ x_3  \end{array} \right) $$
The cohomology of this monad is a rank 2 locally free sheaf $E$ with $c_1(E)=0$ and $c_2(E)=1$. In addition, note that $E$ is not $\mu$-semistable, since $h^0(E(-1))=h^0(\ker\beta(-1))\ne0$. Since, according to \cite[Theorem 3.2]{CMR}, $E(7)$ is regular (in the sense of Castelnuovo--Mumford), hence globally generated, it follows from Proposition \ref{global} that $E(-6)$ is the tangent sheaf of a codimension one distribution of degree 14, i.e. there are codimension one distributions of the form
\begin{equation}\label{ex deg 14}
\sF ~:~ 0 \to E(-6) \to T\p3 \to \IC(16) \to 0
\end{equation}
where $C$ is a smooth curve of degree 161 and genus 1639.

Note, in addition, that $h^2(T_\sF(-16))=h^2(E(-22))=h^1(E(18))$, where the last equality follows from Serre duality. However, $h^1(E(18))=0$ because $E$ is 7-regular. It follows from Proposition \ref{con} that $C$ is connected. Thus, by Theorem \ref{sm-conn}, we conclude that codimension one distributions as the one in display (\ref{ex deg 14}) are not integrable.
\end{Ex}

Another example will be described in Theorem \ref{logStable} below: a codimension one foliation of degree 2 whose tangent sheaf is \emph{strictly $\mu$-semistable} (that is, $\mu$-semistable but not stable), and whose singular scheme consists of a degree 5 curve of genus 2 plus two (possibly embedded) points, counted with multiplicity, see Section \ref{seclog}. 

It would be interesting to determine what is the smallest $\delta$ for which there exists a codimension one distribution of degree $\delta$ whose tangent sheaf is not $\mu$-semistable. In the next three sections, we show that the tangent sheaf of every codimension one distribution of degree at most 2 is $\mu$-semistable; moreover, if the tangent sheaf is locally free, then it is stable. In other words, we so far know that $3\le \delta \le 13$.

For distributions with non-isolated singularities, we have the following general result regarding the stability of their tangent sheaves. First, recall that the \emph{normalization} $E_\eta$ of a torsion free sheaf $E$ on $\p3$ is defined as follows:
$$ E_\eta := \left\{ \begin{array}{l} 
E(-c_1(E)/2),~~ \mbox{if } c_1(E) \mbox{ is even,} \\
E(-(c_1(E)+1)/2),~~ \mbox{if } c_1(E) \mbox{ is odd.}
\end{array} \right. $$
Furthermore, a rank 2 reflexive sheaf $E$ on $\p3$ is $\mu$-semistable if and only if $H^0(E_\eta(-1))=0$, and it is stable if and only if $H^0(E_\eta)=0$ \cite[Lemma 1.2.5, p.165]{OSS}. Moreover, if $c_1(E)$ is odd, then $E$ is $\mu$-stable if and only if it is $\mu$-semistable.

\begin{Prop}\label{non_stable_prop}
Let $\sF$ be a codimension one distribution of degree $d$ on $\P^3$ with nonsplit tangent sheaf, and let $C$ be the pure 1-dimensional subscheme of $\sing(\sF)$ as defined in display \eqref{OZ->OC}. If
\begin{enumerate}
\item[(i)] If $d$ is even, then $T_\sF$ is stable when $2\deg(C) < d^2+d$, and $T_\sF$ is $\mu$-semistable when $2\deg(C)<d^2+3d+2$.
\item[(ii)] If $d$ is odd, then $T_\sF$ is stable when $2\deg(C) < (d+1)^2 $.
\end{enumerate}
\end{Prop}

\begin{proof}
Given a codimension one distribution $\sF$ of degree $d$, let
$$ r:= \min \{t\in\Z ~|~ h^0(T_\sF(t))\ne 0 \} .$$
Note that every nonzero section $\sigma\in H^0(T_\sF(r))$ induces a sub-foliation
$$ \G~:~ 0 \to \O(-r) \to T\p3 \to N_\G \to 0 $$
of degree $r+1$, since the cokernel of the morphism $\O(-r)\to T_\sF$ induced by $\sigma$ is torsion free, see Remark \ref{min twist}. Let $S$ denote the 1-dimensional component of $\sing(\G)$; Lemma \ref{sing sub-dist} implies that $S$ coincides with the zero locus of $\sigma$, which is a curve of degree $c_2(T_\sF(r))$. A result due to Sancho \cite[Theorem 1.1]{Sancho} tells us that
\begin{equation}\label{nonStable}
c_2(T_\sF(r)) = \deg(S) \leq (r+1)^2 + (r+1) + 1 = r^2 + 3r + 3;
\end{equation}
On the other hand, we have 
$$ c_2(T_\sF(r)) = c_2(T_\sF) +(2-d)r + r^2 = d^2 +2 - \deg(C) + (2-d)r + r^2 $$
since $c_2(T_\sF)=d^2+2-\deg(C)$. Substituting in equation \eqref{nonStable} we obtain
\begin{equation}\label{nonStable2}
\deg(C) \ge d^2 - (d+1) r - 1.
\end{equation}

If $T_\sF$ does not split as a sum of line bundles, then $h^0(T_\sF(-1))=0$ by Lemma \ref{coho_non_split}, thus $r\ge0$. Assuming that $d$ is even, note that if $T_\sF$ is not stable, then $h^0(T_\sF((d-2)/2))\ne0$, thus $r\le d/2-1$. Using these inequalities in \eqref{nonStable2}, we obtain
$$ \deg(C) \ge \dfrac{1}{2}(d^2 + d). $$
Similarly, if $T_\sF$ is not $\mu$-semistable, then $h^0(T_\sF((d-2)/2-1)\ne0$, thus $r\le d/2-2$ and we obtain
$$ \deg(C) \ge \dfrac{1}{2}(d^2 + 3d +2). $$

Finally, assume that $d$ is odd; if $T_\sF$ is not stable, then $h^0(\P^3,T_\sF((d-3)/2))\ne0$, so $r\le (d-3)/2$ and we obtain
$$ \deg(C) \ge \dfrac{1}{2}(d+1)^2. $$
\end{proof}

%%%%%%%%%%%%%%%%%%%%%%%%%%%%%%%%%%%%%%%%%%%%%%%%%%%%%%%%%%%%%%%%%%%%%%%%%%%%%%%%%%%%%%%%%%%%%
%%%%%%%%%%%%%%%%%%%%%%%%%%%%%%%%%%%%%%%%%%%%%%%%%%%%%%%%%%%%%%%%%%%%%%%%%%%%%%%%%%%%%%%%%%%%%

\section{Classification of degree 0 distributions} \label{class0}

The classification of degree 0 distributions is well known, see \cite{J} for the integrable case. Below, we provide a new proof.

\begin{Prop}\label{deg0}
If $\sF$ is a codimension one distribution of degree 0, then
\begin{enumerate}
\item[(i)] either $T_\sF=\O(1)\oplus\O(1)$, and its singular scheme is a line;
 \item[(ii)] or $T_\sF=N(1)$ for some null-correlation bundle $N$, and its singular scheme is empty.
\end{enumerate}

\end{Prop}
\begin{proof}
Since $c_1(T_\sF)=2$, the normalization of $T_\sF$ is $T_\sF(-1)$. If $h^0(T_\sF(-1))\neq 0$, then, by Lemma \ref{coho_non_split}, $T_\sF$ is split, hence $T_\sF=\O(1)\oplus\O(1)$. It follows that its singular scheme has pure dimension 1 (since $T_\sF$ is locally free) and degree 1 (by the first formula in display (\ref{eq:length})), hence it must be a line.  

On the other hand, if $h^0(T_\sF(-1)) = 0$, then $T_\sF$ is a stable reflexive sheaf with Chern classes $c_1(T_\sF(-1))=0$ and $ c_2(T_\sF(-1))=1-\deg(C)$. The Bogomolov inequality guarantees that
$$ 0<c_2(T_\sF(-1))=1-\deg(C), $$ 
hence $\deg(C)=0$ and $c_2(T_\sF(-1))=1$. By \cite[Lemma 2.1]{Chang}, $T_\sF(-1)$ must be locally free. However, every stable rank 2 locally free sheaf $N$ with $c_1(N)=0$ and $c_2(N)=1$ is a null-correlation bundle, see \cite[8.4.1]{H2}. It follows that $T_\sF=N(1)$. 
\end{proof}

We observe that both examples can be realized as codimension one distributions. In fact,  for case  (i), since $h^0(T_\sF(-1))=4$, one can choose two linearly independent sections of $H^0(T_\sF(-1))$ to obtain an injective morphism $\O(1)\oplus\O(1) \to T\p3$. Distributions of type  (ii) were already described in Example \ref{ex nc bdl} above, see display \eqref{nc-bdl dual}.

% REPETITION OF EXAMPLE 5.2!!
% Recall that a null-correlation bundle $N$ is defined as the cokernel of a morphism $\O(-1)\to \Omega^1_{\p3}(1)$, given by a non vanishing section in $H^0(\Omega^1_{\p3}(2))$:
% $$ 0 \to \O(-1) \to \Omega^1_{\p3}(1) \to N \to 0 . $$
% Dualizing this exact sequence and twisting the result by $\O(1)$ we obtain
% $$ \sF: 0\rightarrow N(1)\rightarrow T\mathbb{P}^3\rightarrow \O(2)\rightarrow0, $$
% thus providing a distribution of type \emph{(ii)}. 

%%%%%%%%%%%%%%%%%%%%%%%%%%%%%%%%%%%%%%%%%%%%%%%%%%%%%%%%%%%%%%%%%%%%%%%%%%%%%%%%%%%%%%%%%%%%%
%%%%%%%%%%%%%%%%%%%%%%%%%%%%%%%%%%%%%%%%%%%%%%%%%%%%%%%%%%%%%%%%%%%%%%%%%%%%%%%%%%%%%%%%%%%%%

\section{Classification of degree 1 distributions} \label{class1}

Let $\sF$ be a codimension one distribution of degree 1 on $\p3$ with singular scheme $Z$; let $\mathscr{U}$ be the maximal 0-dimensional subsheaf of $\OZ$, and let $C\subset Z$ be the corresponding subscheme of pure dimension 1. The inequality (\ref{eq:BdegZ}) becomes $\deg(C)\leq 3$; using Theorem \ref{P:SLocus}, we obtain
$$ c_2(T_\sF) = 3 - \deg(C) ~~{\rm and}~~ c_3(T_\sF) = 3 - \deg(C) + 2 p_a(C). $$

Note also that the normalization of $T_\sF$ is again $T_\sF(-1)$. It follows that
\begin{itemize}
\item[(i)] either $h^0(T_\sF(-1))\neq 0$, hence, by Lemma \ref{coho_non_split}, $T_\sF=\O(1)\oplus\O$; \label{page20}
\item[(ii)] or $h^0(T_\sF(-1))= 0$, and $T_\sF$ is a stable reflexive sheaf.
\end{itemize}

In the first case, $c_2(T_\sF)=c_3(T_\sF)=0$, and $Z$ has pure dimension 1 because $T_\sF$ is locally free. Additionally, $\deg(Z)=3$ and $p_a(Z)=0$ and 
$$ h^0(I_Z(1)) = h^0(T_\sF(-2)) = 0, $$
hence $Z$ cannot be contained in any plane. We conclude that $Z$ is a (possibly degenerate) twisted cubic curve.

In the second case, Bogomolov inequality forces $c_2(T_\sF)\ge1$, so that $\deg(C)=0,1,2$. Let us analyze each possibility.

If $\deg(C)=0$, so that $C=\emptyset$, it follows that $c_2(T_\sF)=3$ and $c_3(T_\sF)=5$.

If $\deg(C)=1$, so that $C$ is a line and hence $p_a(C)=0$, it follows that $c_2(T_\sF)=2$ and $c_3(T_\sF)=2$.

If $\deg(C)=2$, it follows that $c_2(T_\sF)=1$ and $c_3(T_\sF)=1+2p_a(C)$. Since $T_\sF$ is reflexive, we must have $c_3(T_\sF)\ge0$, hence $p_a(C)\ge0$. On the other hand, since $T_\sF$ is stable and $c_1(T_\sF)=-1$, it follows from \cite[Theorem 8.2(d)]{H2} that $c_3(T_\sF)\le1$, thus actually $p_a(C)=0$ and $C$ is a (possibly degenerate) conic.

We summarize the facts obtained above in the following table.

\begin{table}[h] %\large{\scshape{ }}
\begin{tabular}{|c | c | c | c | c | c | } \hline
$\deg(C) $ & $c_2(T_\sF)$ & $c_3(T_\sF)$ & $T_\sF$ & $\sing(\sF)$ \\ \hline
0          &  3           &  5           & stable  & 5 points \\ \hline
1          &  2           &  2           & stable  & a line and 2 points \\ \hline
2          &  1           &  1           & stable  & a conic and a point \\ \hline
3          &  0           &  0           & split   & a twisted cubic \\ \hline
\end{tabular}
\medskip
\caption{Codimension one distribution of degree 1; all possibilities are realized.}
 \label{table deg 1}
\end{table}

% Let us now look in detail each of the four classes of distributions described above, showing that each one can be realized concretely. Below, in the proof of Lemma \ref{conic+point}, we will use the \emph{spectrum} of a rank 2 reflexive sheaf on $\p3$, so we pause for a moment to recall it.
% \label{spectrum}
% Let $E$ be a rank 2 reflexive sheaf on $\p3$ with $c_1(E)$ equal to 0 or $-1$ satisfying $H^0(E(-1))=0$. Following Hartshorne \cite[Theorem 7.1]{H2}, the spectrum of $E$ is the unique set of integers $(k_1,\dots,k_{c_2(E)})$, in increasing order, satisfying the following two conditions, where $S=\oplus_{i}\mathcal{O}_{p1}(k_i)$:
% \begin{enumerate}
% \item $h^1(E(l))=h^0(\p1,S(l+1))$, for $l\le-1$;
% \item $h^2(E(l))=h^1(\p1,S(l+1))$, for $l\ge-3$ when $c_1(E)=0$ and $l\ge-2$ when $c_1(E)=-1$.
% \end{enumerate}
% Hartshorne establishes various properties for the spectrum of rank 2 reflexive sheaves in \cite[Section 7]{H2}; the most relevant for us are:
% \begin{itemize}
% \item $\sum_i k_i = -c_3(E)/2$ \cite[Proposition 7.3]{H2};
% \item $(k_1,\dots,k_{c_2(E)})$ is connected, that is $k_{i+1}\le k_{i}+1$, except possibly for a gap at 0 \cite[Corollary 7.6]{H2}.
% \end{itemize}

%----------------------------------

\subsection{Distributions with isolated singularities}

\begin{Theorem} \label{generic_deg1}
Let $\sF$ be a codimension one distribution of degree 1 with isolated singularities. Then its tangent sheaf $T_\sF$ is given as the cokernel of a morphism
$$ 0\to \O(-2)\oplus T\P^3(-3) \to  \O(-1)^{\oplus 6} \to T_\sF \to 0, $$
and there exists an unique (up to scalar) morphism $\phi: T_\sF \to T\P^3$. In particular, $\sF$ does not admit sub-foliations of degree $<2$, and possesses sub-foliations of degree $2$ singular along a curve $Y$ of degree $5$ and arithmetic genus $-4$.
\end{Theorem}

\begin{proof}
In the case at hand, we have $c_1(T_\sF)=1$, $c_2(T_\sF)=3$ and $c_3(T_\sF)=5$; the normalized sheaf $(T_\sF)_{\eta}=T_\sF(-1)$ satisfies $c_1((T_\sF)_{\eta})=-1$, $c_2((T_\sF)_{\eta})=3$ and $c_3((T_\sF)_{\eta})=5$.

We observe that $h^0((T_\sF)_{\eta}(1))=h^0(T_\sF)=0$. Indeed, Lemma \ref{coho_non_split} implies that $h^0(T_\sF(-1))=0$ since $T_\sF$ does not split as a sum of line bundles; therefore, a nonzero section of $T_\sF$ induces a sub-foliation $\sG$ of $T_\sF$ of dimension one and degree 1. Since $\sing(\sF)$ is 0-dimensional, it follows from \cite[Corollary 3]{BCL} that $\sing(\sF)$  must be contained on the 0-dimensional component $S_0(\sG)$ of $\sing(\sG)$. However, $\sing(\sF)$ has length 5, while, following \cite{Sancho}, the length of $S_0(\sG)$ is at most 4.
 
This shows that $(T_\sF)_{\eta}$ corresponds to a general point, in the sense of \cite[Lemma 3.10 and the proof of Theorem 3.14]{Chang}, of the moduli space of stable rank 2 reflexive sheaves with Chern classes $c_1=-1$, $c_2=3$ and $c_3=5$. Therefore, it follows from the proof of \cite[Theorem 3.14]{Chang} that $T_\sF$ is given as the cokernel of a morphism of locally free sheaves:
\begin{equation}\label{seq:gen}
0\to \O(-2)\oplus T\P^3(-3) \to  \O(-1)^{\oplus 6} \to T_\sF \to 0.
\end{equation}
One can then check that $h^0(T_\sF(1))=6$. 

Now, we will show that $h^0(T\P^3\otimes T_\sF^{\vee})=1$.
Since $T_\sF^{\vee}=T_\sF(-1)$, twisting the sequence (\ref{seq:gen}) by $T\P^3(-1)$ we get 
\begin{equation}\label{seq:gen2}
0\to T\P^3(-3)\oplus T\P^3 \otimes T\P^3(-4) \to  T\P^3(-2)^{\oplus 6} \to T\P^3\otimes T_\sF^{\vee } \to 0.
\end{equation}
Twisting the sequence (\ref{seq:gen}) by $T\P^3$, taking cohomology  and using Bott's Formulae  we get 
$$0 =H^0(T\P^3(-2))^6 \to H^0(T\P^3\otimes T_\sF^{\vee }) \to 
H^1(T\P^3(-3)) \oplus  H^1(T\P^3 \otimes T\P^3(-4) ) \to H^1(T\P^3(-2))^6 =0. $$
Since $h^1(T\P^3(-3))=0$, we conclude that $h^0(T\P^3\otimes T_\sF^{\vee })= h^1(T\P^3 \otimes T\P^3(-4) )$. Now, to calculate $h^1(T\P^3 \otimes T\P^3(-4) )$
Taking cohomology and using Bott's Formulae we conclude that $h^1(T\P^3 \otimes T\P^3(-4) )=h^2( T\P^3(-4) )=1$. Therefore
$$h^0(T\P^3\otimes T_\sF^{\vee })= h^1(T\P^3 \otimes T\P^3(-4) )=1.$$
\end{proof}

\begin{Ex}\rm
Consider the distribution  $\sF$ with $c_1(T_\sF)=1$,$c_2(T_\sF)=3,$ $c_3(T_\sF)=5$ induced by 
$$ \omega = (z_0^2+z_1^2+z_2^2)dz_3 - (z_3z_0+z_2z_1)dz_0 + (z_2z_0-z_3z_1)dz_1-z_3z_2dz_2 $$
The singular scheme of $\sF$ is $\{ 2[  i:- 1:0:0], 2[i:1:0:0], [0:0:0:1] \} $.
\end{Ex}

%----------------------------

\subsection{Distributions with singular set being a line and two points}

\begin{Prop} 
Let $\sF$ be a codimension one distribution of degree 1 whose singular scheme is the union of a line with two points. Then its tangent sheaf is a stable reflexive sheaf with Chern classes $c_2(T_\sF)=c_3(T_\sF)=2$, and $\sing(\sF)$ cannot be contained in a plane. In addition, $\sF$ possesses a unique sub-foliation of degree $1$ singular along a pair of skew lines (or its degeneration, a nonplanar degree two structure on a line).
\end{Prop}

\begin{proof}
In this case $c_1(T_\sF)=1$, $c_2(T_\sF)=c_3(T_\sF)=2$ and the normalized sheaf  $(T_\sF)_{\eta}=T_\sF(-1)$ satisfies $c_1((T_\sF)_{\eta})=-1$, $c_2((T_\sF)_{\eta})=c_3((T_\sF)_{\eta})=2$. It follows from \cite[Table 2.6.1]{Chang} that
$$ h^0(I_{Z/\p3}(1)) = h^1((T_\sF)_{\eta}(-1)) = 0 ,$$
thus $\sing(\sF)$ cannot be contained in a plane.

Moreover, $h^0(T_\sF)=h^0((T_\sF)_{\eta}(1))=1$, see \cite[Table 2.6.1]{Chang}; according with \cite[Lemma 2.4]{Chang} a nonzero section of $T_\sF$ vanishes along a pair of skew lines (or its degeneration, a nonplanar degree two structure on a line). This section induces a codimension two sub-foliation of degree one.
\end{proof}

\begin{Ex}\label{ex_grau1:5}\rm
Consider the non-integrable distribution $\sF$ with $c_1(T_\sF)=1$,$c_2(T_\sF)=c_3(T_\sF)=2$ induced by 
$$ \omega= (z_0^2+z_1^2)dz_3 - z_3(z_0dz_0+z_1dz_1) + z_1(z_0dz_2-z_2dz_0). $$
The singular scheme of $\sF$ is $\{z_0=z_1=0\}  \cup \{2 [0:0:1:0]\}$.
\end{Ex}

%-------------------------------------

\subsection{Distributions whose singular set is a conic and a point}

In this case, the tangent sheaf must be a stable reflexive sheaf with $c_1(E)=c_2(E)=c_3(E)=1$. Indeed, more is true.

\begin{Theorem} \label{conic+point}
Every stable rank 2 reflexive sheaf $E$ with Chern classes $c_1(E)=c_2(E)=c_3(E)=1$ is the tangent sheaf of a codimension one distribution on $\P^3$ whose singular set is the union of a degree 2 curve of genus 0 with a point not contained in the same plane.  
\end{Theorem}

\begin{proof}
Let $E$ be a stable rank 2 reflexive sheaf on $\P^3$ with Chern classes
$c_1(E)=c_2(E)=c_3(E)=1$. Its normalization is $E_{\eta}=E(-1)=E^{\vee}$, so that
$c_1(E_{\eta})=-1$ and $c_2(E_{\eta})=c_3(E_{\eta})=1$; it follows from \cite[proof of Lemma 9.3, page 166]{H2} that $E_{\eta}$ is given as the cokernel of a morphism
\begin{equation}\label{exsqc}
0 \to \O(-1) \to \op3^{\oplus 3} \to E_{\eta}(1)\simeq E \to 0.
\end{equation}
It follows from Proposition \ref{global} that $E^{\vee}(1)=[E(-1)](1)=E$ is the tangent sheaf of a codimension one distribution on $\P^3$.

% The spectrum of such $E_{\eta}$ is $(-1)$, see  \cite[Proposition 7.3]{H2}. By \cite[Theorem 7.1]{H2} we have $h^2(E_{\eta}(-1))=h^1(\P^1, \mathcal{O}_{\P^1}(-1))=0$. By stability and Serre duality $h^3(E_{\eta}(-2))=h^0(E_{\eta}^{\vee }(-2))=h^0(E_{\eta}^{\vee }(-1))=0.$ Since also $h^0(E_{\eta})=0$, we conclude that $h^1(E_{\eta})=-\chi(E_{\eta}) =0$, see \cite[Theorem 2.3]{H2}. This shows that $E_{\eta}$ is $1$-regular, i.e. $E_{\eta}(1)=E$ is globally generated; in fact, since $h^0(E)=\chi(E)=3$, we have that every stable rank 2 reflexive sheaf on $\P^3$ with Chern classes $c_1(E)=c_2(E)=c_3(E)=1$ is given as the cokernel of a morphism:
% \begin{equation}\label{exsqc}
% 0 \to \O(-1) \to \op3^{\oplus 3} \to E \to 0.
% \end{equation} 

Theorem \ref{P:SLocus} implies that $\deg(C)=2$ and $p_a(C)=0$; according to \cite[Corollary 1.6]{N} every such curve is planar. On the other hand, the exact sequence
$$ 0 \to E \to T\p3 \to I_{Z/\p3}(3) \to 0 $$
yields $h^0(I_Z(1))=h^1(E(-2))=0$, where the second equality is an immediate consequence of the exact sequence (\ref{exsqc}). It follows that $Z$ is not contained in a plane, thus $C$ and the additional point cannot be contained in the same plane.

% We now check that $C$ is reduced. If it is not, since $p_a(C)=0$, it must be given by an exact sequence of the form \cite[Corollary 1.6]{N}
% $$ 0 \to I_{C/\p3} \to I_{\ell/\p3} \to \mathcal{O}_{\ell/\p3}(-1) \to 0, $$
% where $\ell:=C_{\rm red}$. Passing to cohomology, we obtain
% $$ 0 \to H^0(I_{C/\p3}(2)) \to H^0(I_{\ell/\p3}(2)) \to H^0(\mathcal{O}_{\ell/\p3}(1)),$$
% and it follows that $h^0(I_{C/\p3}(2))\ge h^0(I_{\ell/\p3}(2)) - h^0(\mathcal{O}_{\ell/\p3}(1))=5$.
% On the other hand, the exact sequence
% $$ 0 \to E \to T\p3 \to I_{Z/\p3}(3) \to 0 $$
% yields
% $$ H^0(E(-1)) \to H^0(T\p3(-1)) \to H^0(I_{Z/\p3}(2)) \to H^1(E(-1)) .$$
% Since $h^0(E(-1))=h^1(E(-1))=0$, it follows that $h^0(I_{Z/\p3}(2))=h^0(T\p3(-1))=4$, giving a contradiction.
\end{proof}

\begin{Ex}\label{ex_grau1:3}\rm
A generic foliation  $\sF$ with $c_1(T_\sF)=c_2(T_\sF)=c_3(T_\sF)=1$ is induced by 
$$ \omega= (z_0^2+z_1^2+z_2^2)dz_3-z_3(z_0dz_0+z_1dz_1+z_2dz_2).$$
The singular scheme of $\sF$ is $\{z_3=z_0^2+z_1^2+z_2^2=0\}  \cup \{[0:0:0:1]\}$, see \cite[Proposition 3.5.1]{J}.
\end{Ex}

\begin{Ex}\label{ex_grau1:4}\rm
Consider the non-integrable distribution $\sF$ with $c_1(T_\sF)=c_2(T_\sF)=c_3(T_\sF)=1$ induced by 
$$ \omega= (z_0^2+z_1^2+z_2^2)dz_3-z_3(z_0dz_0+z_1dz_1+z_2dz_2)
+z_3(z_0dz_1-z_1dz_0). $$
The singular scheme of $\sF$ is $\{z_3=z_0^2+z_1^2+z_2^2=0\}  \cup \{[0:0:0:1]\}$.
\end{Ex}

The Example \ref{ex_grau1:3} and Example \ref{ex_grau1:4} show two different distributions with the same singular scheme. In both cases, the tangent sheaves are not locally free; this is to be contrasted with \cite{AC}, where the authors provide conditions for locally free distributions to be determined by their singular schemes.  

\begin{Ex}\label{ex_grau1:2}\rm
Consider the non-integrable distribution $\sF$ induced by 
$$ \omega= z_0z_2dz_1+z_2(z_1-2z_3)dz_0+z_1z_2dz_3+z_1(z_3-2z_0)dz_2. $$
The singular set of $\sF$ is $\{z_1=z_2=0\}\cup \{z_2=z_3-2z_0=0\} \cup \{[0:0:1:0]\}$.
\end{Ex}

%----------------------------------------------

\subsection{Distributions with locally free tangent sheaf}
Summarizing what was already argued in the introduction of the present section (precisely, item (i) in page \pageref{page20}), we obtain the following claim.

\begin{Cor}\label{deg1free}
If $\sF$ is a codimension one distribution of degree 1 with locally free tangent sheaf, then $T_\sF=\O(1)\oplus \O$. In addition, $\sing(\sF)$ is a (possibly degenerate) twisted cubic curve not contained in a plane. 
\end{Cor}

% In particular, these distributions  are given on homogeneous coordinates by a 1-form of the type
% \[
% \omega=A_0(z_0,z_1,z_2,z_3)dz_0+A_1(z_0,z_1,z_2,z_3)dz_1+A_2(z_0,z_1,z_2,z_3)dz_2,
% \]
% where $A_j$ are homogeneous of degree $2$ and $\sum_{i=0}^2 z_iA_i=0$.

\begin{Ex} \label{ex_grau1:1}
As explained in the end of Section \ref{ss:cohomology}, every distribution with $T_\sF=\O(1)\oplus \O$ is given in homogeneous coordinates by a 1-form of the type
$$ \omega = A_0(z_0,z_1,z_2,z_3)dz_0 + A_1(z_0,z_1,z_2,z_3)dz_1 + A_2(z_0,z_1,z_2,z_3)dz_2 , $$
where $A_j$ are homogeneous of degree $2$ and $\sum_{i=0}^2 z_iA_i=0$. So consider, in particular, the non-integrable distribution $\sF$ induced by the 1-form:
$$ \omega= z_0z_2dz_1+z_2(z_3-z_1)dz_0-z_0z_3dz_2.$$
Its singular set is a chain of three lines
$$ \sing(\sF) = \{z_0=z_2=0\}\cup \{z_0=z_3-z_1=0\} \cup \{z_2=z_3=0\}. $$
Note that $A_0=z_0z_2$, $A_1=z_2(z_3-z_1)$ and $A_2=z_0z_3$ are the $2\times2$ minors of the $2\times3$ matrix
$$ \left( \begin{array}{cc}
z_0 & 0 \\ 0 & z_2 \\ z_3 & z_3-z_1
\end{array} \right), $$
showing that $\sing(\sF)$ is indeed a (degenerate) twisted cubic curve.
\end{Ex}

%%%%%%%%%%%%%%%%%%%%%%%%%%%%%%%%%%%%%%%%%%%%%%%%%%%%%%%%%%%%%%%%%%%%%%%%%%%%%%%%%%%%%%%%%%%%%
%%%%%%%%%%%%%%%%%%%%%%%%%%%%%%%%%%%%%%%%%%%%%%%%%%%%%%%%%%%%%%%%%%%%%%%%%%%%%%%%%%%%%%%%%%%%%

\section{Classification of degree 2 distributions}\label{class2}

Let $\sF$ be a codimension one distribution of degree 2 on $\p3$ with singular scheme $Z$; let $\mathscr{U}$ be the maximal 0-dimensional subsheaf of $\OZ$, and let $C\subset Z$ be the corresponding subscheme of pure dimension 1. The inequality (\ref{eq:BdegZ}) becomes $\deg(C)\leq 7$; using Theorem \ref{P:SLocus}, we obtain
\begin{equation}\label{c2 c3}
c_2(T_\sF) = 6 - \deg(C) ~~{\rm and}~~ c_3(T_\sF) = 18 - 4\deg(C) + 2 p_a(C).
\end{equation}

Since $c_1(T_\sF)=0$, $T_\sF$ is already normalized. It follows that:
\begin{itemize}
\item[(i)] either $h^0(T_\sF(-1))\neq 0$, hence, by Lemma \ref{coho_non_split}, $T_\sF$ splits as a sum of line bundles;
\item[(ii)] or $h^0(T_\sF(-1))= 0$, and $T_\sF$ is a $\mu$-semistable reflexive sheaf.
\end{itemize}

In the first case, $T_\sF=\O(-1)\oplus\O(1)$, hence $c_2(T_\sF)=-1$ and $c_3(T_\sF)=0$; it follows that $Z=C$ is a curve of degree 7 and arithmetic genus 5.

When $T_\sF$ is $\mu$-semistable, we must have $c_2(T_\sF)\ge0$. We now study each of six possibilities: $0\le c_2(T_\sF)\le 6$, or, equivalently, $0\le\deg(C)\le6$.

If $\deg(C)=0$, so that $C=\emptyset$, it follows that $c_2(T_\sF)=6$, $c_3(T_\sF)=20$, and $\sing(\sF)$ consists of 20 points. In addition, Theorem \ref{StGDist} implies that $T_\sF$ is stable. 

If $\deg(C)=1$, so that $C$ is a line and hence $p_a(C)=0$, it follows that $c_2(T_\sF)=5$, $c_3(T_\sF)=14$, hence $\sing(\sF)$ consists of a line plus 14 points. In addition, Proposition \ref{non_stable_prop} implies that $T_\sF$ is stable.

If $\deg(C)=2$, then $p_a(C)\le0$ \cite[Corollary 1.6]{N}. Since $T_\sF$ is reflexive and comparing with the formulas in Theorem \ref{P:SLocus}, it follows that $c_2(T_\sF)=4$, and $0\le c_3(T_\sF)=10+2p_a(C)\le 10$. In addition, Proposition \ref{non_stable_prop} implies that $T_\sF$ is stable.

If $\deg(C)=3$, then $p_a(C)\le1$ by a classical result, see \cite{H3}. Since $T_\sF$ is reflexive and comparing with the formulas in Theorem \ref{P:SLocus}, it follows that $c_2(T_\sF)=3$, and $0\le c_3(T_\sF)=6+2p_a(C)\le 8$. 

When $\deg(C)=4$, we first exclude the possibility of $C$ being a plane quartic: in this case, $p_a(C)=3$ then $c_3(T_\sF)=8$ which is impossible for a $\mu$-semistable rank 2 reflexive sheaf with $c_1(T_\sF)=0$ and $c_2(T_\sF)=2$, cf. \cite[Theorem 8.2(a)]{H2}. Hence $C$ is not planar, and by \cite[Theorem 3.3]{H3}, $p_a(C)\le1$, hence $0\le c_3(T_\sF)=2+2p_a(C)\le 4$. 

If $\deg(C)=5$, then $T_\sF$ is a $\mu$-semistable rank 2 reflexive sheaf with $c_2(T_\sF)=1$; it follows from \cite[Theorem 8.2(a)]{H2} that $c_2(T_\sF)=0,2$.

Finally, if $\deg(C)=6$, it follows that $c_2(T_\sF)=0$ hence actually $T_\sF=\O\oplus\O$, cf. \cite[Remark 3.1.1]{H2}, since $T_\sF$ is $\mu$-semistable; it follows that $Z=C$ is a curve of degree 6 and arithmetic genus 3.

We summarize in the results obtained above in Table \ref{deg 2 table}. 

\begin{table}[h] %\large{\scshape{ }}
\begin{tabular}{|c | c | c | c | }  \hline
$\deg(C) $ & $c_2(T_\sF)$ & $c_3(T_\sF)$         & $T_\sF$    \\ \hline
0          &  6           &  {\bf 20}            & stable     \\ \hline
1          &  5           &  14                  & stable     \\ \hline
2          &  4           &  0, 2, 4, 6, 8, 10   & stable     \\ \hline
3          &  3           &  0, 2, 4, 6, {\bf 8} & $\mu$-semistable \\ \hline
4          &  2           &  {\bf 0}, 2, {\bf 4} & $\mu$-semistable \\ \hline
5          &  1           &  {\bf 0}, {\bf 2}    & $\mu$-semistable \\ \hline
6          &  0           &  {\bf 0}             & split      \\ \hline
7          &  -1          &  {\bf 0}             & split      \\ \hline
\end{tabular}
\medskip
\caption{Codimension one distribution of degree 2; values in boldface are realized.}
\label{deg 2 table}
\end{table}

It is not clear to us whether there actually exists a codimension one distribution of degree 2 for each of the possible values of the second and third Chern classes. The pairs $(c_2(T_\sF),c_3(T_\sF))=(3,8)$, $(2,4)$, and $(1,2)$ are realized by well-known families of foliations, namely \emph{rational} and \emph{logarithmic} foliations. We will study them in detail in Section \ref{rat-log} below. 

% The tangent sheaf of a codimension 1 distribution $\sF$ of degree 2 is strictly $\mu$-semistable if and only if $\sF$ admits a sub-foliation $\sG$ of degree 1, induced by a non-trivial section $\sigma\in H^0(T_\sF)$. This means that we can set $k=0$ in the sequences \eqref{v1} and \eqref{v2}, providing
% \begin{equation}\label{v2.2}
% 0 \to \O(-4) \to  N_\sG^\vee \to \IC \to 0 ~~{\rm and}~~
% 0 \to \OC \to \omega_{C} \to \mathscr{V} \to 0.
% \end{equation}
% In particular, it follows that $h^0(\omega_{C})\ne0$.

% If $d=2$, then cohomology sequence associated with short exact sequence in display (\ref{eq:Dist p3}) is given by
% \begin{equation} \label{sqc deg 2 b}
% 0 \to H^0(T_\sF) \to H^0 (T\p3) \to H^0(\IZ(4)) \to H^1(T_\sF) \to 0.
% \end{equation}
% Therefore, we conclude that the tangent sheaf of a degree 2 distribution is stable if and only if the map $H^0(T\p3) \to H^0(\IZ(4))$ is injective. 

For the moment, we observe that the tangent sheaves of codimension one foliations of degree 2 with invariants $(c_2(T_\sF),c_3(T_\sF))=(3,8)$, $(2,4)$ are stable.

\begin{Lemma}\label{(3,8)}
If $\sF$ is a codimension one foliation  of degree 2 such that $c_2(T_\sF)=3$ and $c_3(T_\sF)=8$, then $T_\sF$ is stable, and $C$ is a plane cubic curve.
\end{Lemma}
\begin{proof}
Note that $C$ is a curve of degree 3 and genus 1, hence it must be a plane cubic, cf. \cite[proof of Proposition 3.1]{N}; in particular, we have that $\omega_C=\OC$.

Clearly, $T_\sF$ does not split as a sum of line bundles, so $h^0(T_\sF(-1))=0$ by Lemma \ref{coho_non_split}. If $T_\sF$ is not stable, then a nonzero section $\sigma\in H^0(T_\sF)$ induces  a sub-foliation $\sG$ of degree 1 whose 1-dimensional component of its singular scheme is, by Lemma \ref{sing sub-dist}, a curve of degree 3 and genus $-1$, hence $\sing(\sG)\ne C$. Using the integrability of $\sF$, it follows from \cite[Theorem 1]{Mol} that $C$ must be invariant by $\sG$; however, every nontrivial section
$$ \nu\in H^0(\omega_C^\vee\otimes T_{\sG}|_C)=H^0(\OC), $$ 
is nowhere vanishing, thus contradicting the Jouanolou--Esteves--Kleiman result mentioned in Remark \ref{JEK}.

% it follows that the sheaf $\mathscr{V}$ in the rightmost sequence in display \eqref{v2.2} vanishes, meaning that the sequences in the displays \eqref{v3} and \eqref{v4} simplify to 
% $$ 0 \to \inext^1(N_\sG,\O) \to \omega_Y \to \mathscr{U} \to 0 , $$
% thus $\sing(\sG)=Y$, which is a curve of degree 3 and genus $-1$, hence $\sing(\sG)\ne C$. Using the integrability of $\sF$, it follows from \cite[Theorem 1]{Mol} that $C$ must be invariant by $\sG$; however, every nontrivial section
% $$ \nu\in H^0(\omega_C^\vee\otimes\sG|_C)=H^0(\OC), $$ 
% is non-vanishing, thus contradicting the Jouanolou--Esteves--Kleiman result mentioned in Remark \ref{JEK}.
\end{proof}

\begin{Lemma}\label{(2,4)}
If $\sF$ is a codimension one foliation  of degree 2 such that $c_2(T_\sF)=2$ and $c_3(T_\sF)=4$, then $T_\sF$ is stable, and $C$ is an elliptic quartic curve.
\end{Lemma}

\begin{proof}
According to \cite[Table 2.12.2 and Table 2.16.1]{Chang}, every $\mu$-semistable rank 2 reflexive sheaf $E$ with $c_1(E)=0$, $c_2(E)=2$, and $c_3(E)=4$ must have $h^1(E(-2))=0$ and $h^2(E(-2))=2$. Therefore, using the sheaf sequence in display \eqref{eq:Dist p3} twisted by $\O(-2)$, we have $h^0(\IZ(2))=h^1(T\sF(-2))=0$, and $h^1(\IZ(2))=h^2(T\sF(-2))=2$. Next, we use the sequence in display \eqref{OZ->OC} twisted by $\O(2)$ to obtain the following exact sequence
$$ 0 \to H^0(\IC(2)) \to H^0(\mathscr{U}) \to H^1(\IZ(2)) \to H^1(\IC(2)) \to 0. $$
Since, by Theorem \ref{P:SLocus}, $h^0(\mathscr{U})=c_3(T_\sF)=4$, we conclude that $h^0(\IC(2))\ge2$, so $C$ must be contained in a quadric hypersurface. The same Theorem \ref{P:SLocus} tells us that $C$ is a degree 4 curve of genus 1, so \cite[Proposition 3.5]{H3} implies that $C$ is an elliptic quartic curve.

In particular, we have that $\omega_C=\OC$. With this in mind, the proof of stability of $T_\sF$ is exactly the same as for the stability part of Lemma \ref{(3,8)}.
\end{proof}

\begin{Ex}\label{ex_grau2:2}\rm
Consider the non-integrable distribution $\sF$ induced by the 1-form
$$ \omega = (z_0^3+z_1^3)dz_3 + (z_1^2z_2-z_0^2z_3)dz_0 - z_3z_1^2dz_1 - z_0z_1^2dz_2. $$
As a variety, the singular set is just the line $\{z_0=z_1=0\}$; however, note that the ideal generated by the coefficients of $\omega$ is contained in the ideal generated by the monomials $z_0^2$, $z_1^2$. This means that, in this case, the curve $C$ is a line of multiplicity 4 with genus 1, so that $c_2(T_\sF)=2$ and $c_3(T_\sF)=4$. 
\end{Ex}

\begin{Obs}\label{rmk 9.4}\rm
If $\sF$ is a codimension one distribution of degree 2 such that $c_2(T_\sF)=1$ and $c_3(T_\sF)=2$, then $T_\sF$ must be strictly $\mu$-semistable, because there are no stable rank 2 reflexive sheaves with these invariants \cite[Lemma 2.1]{Chang}. In addition, $C$ is a degree 5 curve of genus 2.
\end{Obs}

Let us now focus on codimension one distributions $\sF$ of degree 2 with locally free tangent sheaf. We say that a curve of degree $m$ and genus $p$ is \emph{reduced up to deformation} if it lies in a component of the Hilbert scheme ${\rm Hilb}_{m,g}$ whose general point corresponds to a reduced curve. 

\begin{Theorem}\label{classification_degre_tow}
Let $\sF$ be a codimension one distribution of degree $2$ with locally free tangent sheaf $T_\sF$, and such that $\sing(\sF)$ is reduced, up to deformation. Then:
\begin{enumerate}
\item  $T_\sF$ splits as a sum of line bundles and 
\begin{enumerate}

\item either $T_\sF=\O(1)\oplus \O(-1)$, and $\sing(\sF)$ is a connected curve of degree $7$ and arithmetic genus $5$.

\item or $T_\sF=\O\oplus \O$, and $\sing(\sF)$ is a connected curve of degree $6$ and arithmetic genus $3$.

\end{enumerate}

\item  $T_\sF$ is stable, and:

\begin{enumerate}
\item either $T_\sF$ is a null-correlation bundle, and $\sing(\sF)$ is a connected curve of degree $5$ and arithmetic genus $1$; in addition, $\sF$ possesses  sub-foliations of degree $2$ which are singular along two skew lines;
 
\item or $T_\sF$ is an instanton bundle of charge $2$, and $\sing(\sF)$ is the disjoint union of a line and a twisted cubic, up to deformation; in addition, $\sF$ possesses sub-foliations of degree $2$ which are singular along three skew lines.  
\end{enumerate}
\end{enumerate}

In addition, this result is effective, in the sense that there exist injective morphisms $\phi : T_\sF \to T\P^3$ with torsion free cokernel for each of the possibilities listed above.  
\end{Theorem}

\begin{proof}
If $T_\sF$ splits as a sum of line bundles, then only the two possibilities listed in the statement of the theorem occur. The degree and genus of the singular set are easily computed via formula (\ref{c2 c3}). 

If $T_\sF$ does not split as a sum of line bundles, then, according to Table \ref{deg 2 table}, we must have $1\le c_2(T_\sF)\le4$. 

If either $c_2(T_\sF)=3$ or $c_2(T_\sF)=4$, then $\sing(\sF)$ cannot be reduced up to deformation: in the first case, $\sing(\sF)$ would be a degree 3 curve of genus $-3$, while in the second case $\sing(\sF)$ would be a degree 2 curve of genus $-5$, and there are no reduced curves with these degrees and genera.

Next, suppose that $T_\sF$ is strictly $\mu$-semistable with $c_2(T_\sF)=2$; it follows that
$$ h^0(\OZ(-1))=h^1(\IZ(-1))=h^2(T_\sF(-5))=h^1(T_\sF(1)), $$
where the last equality is given by Serre duality. According to \cite[Table 2.16.1]{Chang}, $h^1(T_\sF(1))=2$, hence $Z$ cannot be reduced. Therefore, we conclude if $\sing(\sF)$ is reduced up to deformation, then $T_\sF$ must be stable.

Finally, we remark that every $\mu$-semistable locally free sheaf with $c_1(T_\sF)=0$ and $c_2(T_\sF)=1$ is in fact stable, cf. \cite[Lemma 2.1]{Chang}.

Since we have eliminated all other possibilities, it is now enough to establish the effectiveness part of the Theorem. 

The existence of distributions of type (1a) or (1b) is clear, and the characterization of their singular loci is an easy calculation together with the application of Theorem \ref{con}.

Now let $N$ be a null-correlation bundle; just setting $k=1$ in Example \ref{ex nc bdl}, we obtain a distribution of the form
\begin{equation}\label{eq:Null}
\D: 0\rightarrow N\rightarrow T\mathbb{P}^3\rightarrow I_{Z/\P^3}(4)\rightarrow0
\end{equation}
where $Z$ is a curve of degree $5$ and arithmetic genus 1. Since 
$h^2(\P^3,T_\sF(-2-d))=h^2(\P^3,N(-4))=0$, Theorem \ref{con} allows us to conclude that $Z$ is connected. Every section $\sigma\in H^0(N(1))$ vanishes along two skew lines, and hence it yields a sub-foliation
of degree 2 with the desired properties.

Finally, let $F$ be an instanton bundle of charge $2$; recall that these are given as cohomologies of monads of the form 
$$ \O(-1)^{\oplus 2} \to \O^{\oplus 6} \to \O(1)^{\oplus 2}; $$
see also \cite[section 9]{H1}. We will now show that there exist injective morphisms $\phi: F \to T\P^3$ whose cokernel is torsion free. First, note that twisting the Euler sequence by $F^\vee\simeq F$ and taking cohomology we get
$$ H^0(F)=0\to H^0(F(1))^{\oplus 4} \to H^0(F\otimes T\p3)\to H^1(F) \to 0=H^1(F(1))^{\oplus 4}, $$
so that
$$ \dim\Hom(F,T\p3) = h^0(F\otimes T\p3) = 4\cdot h^0(F(1)) + h^1(F) = 10. $$

Since $\dim\Ext^1(F,\O)=h^1(F)=2$, there exists an indecomposable locally free sheaf $\widehat{F}$ fitting into the sequence 
$$ 0\to \O \to  \widehat{F} \to F\to 0. $$
Now, applying the functor $\Hom(\cdot,\O(1))$ to the previous sequence, we get
$$
0\to \Hom(F,\O(1)) \to \Hom(\widehat{F},\O(1)) \stackrel{\psi}{\to} \Hom(\O,\O(1)) \to 0, $$ 
since $\Ext^1(F,\O(1)) \simeq H^1(F(1))=0$. Consider $s_0,s_1,s_2,s_3 \in \Hom(\widehat{F},\O(1))$ such that $\psi(s_i)=z_i \in \Hom(\O,\O(1))$, for $i=0,\dots,3$. Define the morphism  $\widehat{\rho }: \widehat{F} \to \O(1)^{\oplus 4} $ by $\widehat{\rho}=(s_0,s_1,s_2,s_3)$ and $\rho =(z_0,z_1,z_2,z_3)$. We have the commutative diagram
$$ \xymatrix{
\O  \ar@{=}[d] \ar[r] &\widehat{F} \ar[d]^-{\widehat{\rho }} \\
\O \ar[r]^-{\rho }   & \O(1)^{\oplus 4}
} $$
Our first task is to show that $\widehat{\rho}$ is injective; indeed, consider, for each $i=0,\dots,3$, the diagram 
$$ \xymatrix{
&   &   &  0 \ar[d] &    0 \ar[d]    & \\
&   &   &   \O \ar@{=}[r]  \ar[d]^-{s_i}   &  \O  \ar[d]      & \\
&0\ar[r]& T_\sF(1) \ar@{=}[d] \ar[r]   & \widehat{F}^{\vee }(1) \ar[r] \ar[d]   &\O(1)\ar[r] \ar[d]&0 \\
&0\ar[r]& T_\sF(1)  \ar[r]   & \G_i \ar[r] \ar[d]   &\mathcal{O}_{P_i}(1)\ar[r] \ar[d]&0 \\
& &     & 0      &  0 &
} $$
where $P_i=\{z_i=0\}$. We get
$$  0\to \G_i^{\vee } \to F(-1)\to \inext^1(\mathcal{O}_{P_i}(1),\O) \to 
\inext^1(\G_i,\O) \to 0 ; $$
we also conclude that $\inext^p(\G_i,\O)=0$, for $p=2,3$, so that
$$ 
\sing(\G_i) = \supp\left(\inext^1(\G_i,\O)\right)= \{s_i=0\}. 
$$
In addition, since 
$$ \supp\left(\inext^1(\G_i,\O)\right) \subset  \supp\left(\inext^1(\mathcal{O}_{P_i}(1),\O)\right) ,$$
it follows that $\{s_i=0\} \subset \{z_i=0\}$, hence 
$\cap_{i=0}^3\{s_i=0\}\subset \cap_{i=0}^3\{z_i=0\}= \emptyset$, thus $\widehat{\rho}$ is injective. 

Letting $K:=\coker\widehat{\rho}$, we have the following diagram after applying the Snake Lemma:
\begin{equation}
\xymatrix{
        &                      & 0 \ar[d]                                     & 0 \ar[d]                   & \\
0\ar[r] & \O \ar@{=}[d] \ar[r] & \widehat{F} \ar[d]^-{\widehat{\rho }} \ar[r] & F \ar[r] \ar[d]^-{\phi }   & 0 \\
0\ar[r] & \O \ar[r]^-{\rho }   & \O(1)^{\oplus 4} \ar[r] \ar[d]               & T\mathbb{P}^3\ar[r] \ar[d] & 0 \\
        &                      & K \ar@{=}[r] \ar[d]                & K \ar[d]         & \\
        &                      & 0                                            & 0                          &
}
\end{equation}
  with an induced injective morphism $\phi: F\to T\P^3$.
 Therefore, a    generic element of  $\Hom(F,T_{\P^3}) $ is injective. 
 
 %It is now enough to argue that $K$ is torsion free.

Assume that $K$ does have torsion, and let $T(K)$ be its maximal torsion subsheaf, given by the kernel of the natural morphism $K\to K^{\vee\vee}$. The morphism $\phi:F\to T\P^3$ factors through a reflexive sheaf $G$, giving a monomorphism $\phi':G\to T\P^3$ such that $\coker(\phi')=K/T(K)$ is torsion free. This means that $G$ is the tangent sheaf of a codimension one distribution
$$ \G ~:~ 0 \to G \stackrel{\phi'}{\to} T\p3 \to K/T(K) \to 0, $$
which is called the \emph{saturation} of $\phi:F\to T\p3$. Dualizing and taking the second exterior powers of the morphisms $\phi$ and $\phi'$ we obtain the diagram
$$ \xymatrix{ \Omega^2 \ar[r]^\omega \ar[dr]^{\omega'} & \det(F^\vee)=\O \\ & \det(G^\vee)=\O(d'-2) \ar[u]^{f} } $$
where $d'$ is the degree of $\G$. We therefore obtain two 1-forms $\omega\in H^0(\Omega^1(4))$ and $\omega\in H^0(\Omega^1(d'+2))$, and a polynomial $f\in H^0(\O(2-d'))$ such that $\omega=f\omega'$. In particular, $d'=0,1$ and
$$ \sing(\G)= \{\omega'=0\} \subset \{\omega=0\}. $$

We argue that $G=T_\G$ must be a locally free sheaf. Indeed, recall that, since $\phi: F\to T\P^3$ is an injective morphism of locally free sheaves, then its degeneracy scheme, given by zero locus of the 1-form $\omega$, has codimension at most 2, see \cite[Chapter, section 4]{ACGH}. Thus if $G$ is not locally free, then $\sing(\G)$ contains a component of codimension 3, contradicting the fact that $\codim\{\omega=0\}\le2 $.

Applying the functor $\Hom(-,T_{\P^3})$ to the exact sequence
\begin{equation}\label{saturado2}
0\to  F  \stackrel{\sigma}{\longrightarrow} G \to T(K) \to 0.
\end{equation}
we get
$$ 0 \rightarrow \Hom(G,T_{\P^3}) \stackrel{\circ\sigma}{\to} \Hom(F,T_{\P^3}) \rightarrow \cdots $$
It follows that an injective morphism $\phi\in\Hom(F,T_{\P^3})$ has a torsion free quotient if $\phi\notin\img(\circ\sigma)$. Therefore, it is enough to show that $\dim\Hom(F,T_{\P^3})$ is strictly greater than $\dim\Hom(G,T_{\P^3})$.

% We   have an exact sequence 
% \begin{equation}\label{saturado2}
% 0\to  F  \stackrel{\sigma}{\to} G \to T(K) \to 0.
% \end{equation}
% Now, we apply  the functor $\Hom(-,T_{\P^3})$  to the exact sequence (\ref{saturado2}) and we get
% %According to the Cerveau--Lins Neto classification of codimension 1 %foliations of degree 2 on $\p3$ \cite{CL}, the (non-saturated) %distribution induced by the 1-form $\omega$ cannot be integrable; %since its saturation $\omega'$ coincides with $\omega$ outside of the %surface $S$, we conclude that $\omega'$ is not integrable either. 
% $$
% 0 \rightarrow  \Hom(G,T_{\P^3})  \stackrel{J}{\to} \Hom(F,T_{\P^3}) \rightarrow \cdots
% $$

% Therefore, an injective morphism $\phi \in \Hom(F,T_{\P^3}) $ has   a torsion free quotient if  $\phi \notin  Im(J)$. Thus, it is   enough
% to  show that  $ \dim \Hom(F,T_{\P^3}) $ is strictly greater than   $\dim \Hom(G,T_{\P^3})$.
% First of all, we observe that 
% $$
% \dim \Hom(F,T_{\P^3})=h^0(F\otimes T_{\P^3})=10,
% $$
% see  Proposition  \ref{space-inst}. 

As mentioned above, the degree of the saturated distribution $\G$ is either 1 or 0; let us analyze both cases. 

If $\G$ has degree 1, then $G=\O\oplus \O(1)$ since it must be locally free. However, $h^0(F)=0$, thus there can be no injective morphism $\sigma: F\to \O\oplus \O(1)$. In fact, every morphism $\sigma=(\sigma_1,\sigma_2):F \to \O\oplus \O(1)$ satisfies $\sigma_1=0$, since $\sigma_1\in H^0(F^{\vee})=H^0(F)=0$.

If $\G$ has degree 0, we have from Proposition \ref{deg0} that either $G=N(1)$, where $N$ is a null-correlation bundle, or $G =\O(1)\oplus\O(1)$. Note that
$$
\dim \Hom(N(1), T_{\P^3})=h^0(N(-1)\otimes T_{\P^3})=1, 
$$
see the proof of Proposition \ref{space_null} below, and   
$$
\dim \Hom(\O(1)\oplus\O(1),T_{\P^3})=h^0(  T_{\P^3}(-1)^{\oplus 2})=8.
$$
This shows us that  $\dim \Hom(G,T_{\P^3}) < \dim \Hom(F,T_{\P^3})=10$ in both cases (the last equality is established in the proof of Proposition \ref{space-inst} below), as desired.

We therefore conclude that there exists an injective morphism  $\phi\in\Hom(F,T_{\P^3})$  with torsion free quotient  
$$  0 \to F  \stackrel{\phi}{\to}  T\p3 \to \IC(4) \to 0 $$
inducing a  holomorphic distribution whose tangent sheaf is precisely the instanton bundle of charge 2; let $C$ denote its singular locus.

Theorem \ref{P:SLocus} implies that its singular set is a curve $C$ of degree 4 and arithmetic genus -1; the Hilbert scheme of such curves has 3 irreducible components, described as follows, see \cite[Table II]{NS}: the general point of each component corresponds to: (1) the disjoint union of a twisted cubic and a line; (2) the disjoint union of two conics; (3) a so-called \emph{extremal curve}, that is, a curve $D$ whose \emph{Rao function}
$$ n \mapsto h^1(I_{D/\p3}(n)) $$
is the largest possible with respect to its degree and genus, see \cite[Definition 9.2]{HSc} for further details. In particular, an extremal curve $D$ satisfies $h^0(I_{D/\p3}(1))=0$ and $h^0(I_{D/\p3}(2))\ge2$, cf. \cite[proof of Proposition 9.5, p. 5672]{HSc}.

However, the exact sequence in cohomology
$$ \cdots \to H^0(T\p3(-2)) \to H^0(\IC(2)) \to H^1(F(-2)) \to \cdots $$
implies that $h^0(\IC(2))=0$, thus $C$ cannot be extremal. It cannot be the union of two disjoint conics either, because in this case $h^0(I_{C/\P^3}(2))=1$ since $C$ is contained in the union of two planes. 

It follows that $C$ must lie in the closure of the first component, so it is either the disjoint union of a line and a twisted cubic, or some degeneration of such curve.

This completes the proof of Theorem \ref{classification_degre_tow}.\end{proof}

%%%%%%%%%%%%%%%%%%%%%%%%%%%%%%%%%%%%%%%%%%%%%%%%%%%%%%%%%%%%%%%%%%%%%%%%%%%%%%%%%%%%
%%%%%%%%%%%%%%%%%%%%%%%%%%%%%%%%%%%%%%%%%%%%%%%%%%%%%%%%%%%%%%%%%%%%%%%%%%%%%%%%%%%%

\section{The tangent sheaves of foliations} \label{rat-log}

In this section we study properties of the tangent sheaves of two well known families of codimension one foliations on projective spaces, namely rational and logarithmic foliations.

%----------------------------------------------

\subsection{Tangent sheaves of rational foliations} 

We denote by $\mathcal{R}(a,b)$, the set of codimension one foliations represented by 1-forms of the type
\begin{equation}\label{eq:rat}
\o= p\psi d\varphi- q \varphi d\psi
\end{equation}
where $\psi\in H^0(\O(a))$ and $\varphi\in H^0(\O(b))$ ($1\leq a\leq b$) have no common factor, and $(p,q)$ are the unique relatively prime integers such that $a\, q =b\, p$. Clearly, $\o\in H^0(\Omega^1_\p3(a+b-2))$. The leaves are the fibres of the rational map $\varphi^p/\psi^q:\P^3\dashrightarrow \P^1$.

We also remark that if $a\nmid b$, then $\mathcal{R}(a,b)$ can be regarded as the open subset of $\mathbb{P}H^0(\O(a))\times\mathbb{P}H^0(\O(b))$ consisting of those pairs $(\psi,\varphi)$ with no common factors; in particular:
$$ \dim \mathcal{R}(a,b) = {{a+3}\choose{3}} + {{b+3}\choose{3}} - 2. $$ 
If $a|b$ and $a\ne b$, then $\mathcal{R}(a,b)$ is a projective bundle over 
$\mathbb{P}H^0(\O(a))$, with fiber over $\psi\in\mathbb{P}H^0(\O(a))$ equal to $\mathbb{P}H^0(\O(b))/\langle \psi^{b/a}\rangle$.

Finally, if $a=b$, then $\mathcal{R}(a,a)$ is isomorphic to the grassmannian  of pencils \linebreak $\langle\psi,\phi\rangle\in{\rm Grass}(2,H^0(\O(a))$. 

Further details on rational foliations can be found in \cite{LV}.

\begin{Theorem}\label{racionais} 
Let $\mathcal{F}\in \mathcal{R}(a,b)$ be a rational foliation of degree $d=a+b-2$ on $\P^3$. Then:
\[\begin{array}{lll}
c_1(T_\sF) &=& 2-d=4-(a+b) \\
c_2(T_\sF) &=& d^2+2-ab=a^2+b^2+ab-4(a+b)+6 \\
c_3(T_\sF) &=& (a+b-2)^3+2(a+b-2)^2+2(1-ab)(a+b-2)
\end{array}
\]
In addition:
\begin{enumerate}
\item[(i)] if $d=0$, then $T_\sF$ is isomorphic to $\O(1)\oplus\O(1)$;
\item[(ii)] if $d\geq 1$, then $T_\sF$ is stable.
\end{enumerate}
\end{Theorem}

Note that the tangent sheaf of a rational foliation is locally free only when $d=0$, since $c_3(T_\sF)\ne0$ otherwise.

We will need the following elementary fact in the proof of the Theorem \ref{racionais}.

\begin{Obs}\label{cota_Log}\rm
Let $d_1,\dots,d_r$ be positive  integers such that $\sum_{i=1}^rd_i=d+2$. Then
\[
\sum_{1 \leq i<j\leq r} d_id_j\leq \frac{r(r-1)}{2r^2}(d+2)^2.
\]
In fact, 
apply the  method of Lagrange multipliers for the function $f(x_1,\dots,x_r)=\sum_{1 \leq i<j\leq r} x_ix_j$ under the conditions $\sum_{i=1}^rx_i=d+2$ and $x_i>0$ for all $i=1,\dots,r.$
\end{Obs}

\begin{proof}
Let $Z$ be the singular scheme of $\sF$; let $\mathscr{U}$ be the maximal 0-dimensional subsheaf of $\OZ$, and let $C\subset Z$ be the corresponding subscheme of pure dimension 1. For rational foliations, $C$ is a complete intersection curve $\{\psi=\varphi=0\}$, so that $\deg(C)=ab$ and $2p_a(C)-2= ab(a+b-4)$. The formulas for the Chern classes of the tangent sheaf now follow immediately from Theorem \ref{P:SLocus}.

According to Proposition \ref{deg0} above, the tangent sheaf of a codimension one distribution of degree 0 on $\p3$ either splits as a sum of line bundles, or is isomorphic to a null-correlation bundle. Since the latter is never integrable, the first item follows immediately.

When $d=1$, Table \ref{table deg 1} also guarantees that $T_\sF$ is stable. The Chern classes of rational foliations of degree 2 are $(c_2(T_\sF),c_3(T_\sF))=(3,8)$, in the case $(a,b)=(1,3)$, and $(c_2(T_\sF),c_3(T_\sF))=(2,4)$, in the case $a=b=2$. The stability of the tangent sheaves was  shown in Lemma \ref{(3,8)} and Lemma \ref{(2,4)}, respectively.

Next, note that Remark \ref{cota_Log} implies that 
$$ 2\deg(C)\leq \frac{(d+2)^2}{2}. $$
Since $(d+2)^2/2<d^2+d$ when $d\ge 4$ and $(d+2)^2/2<(d+1)^2$ when $d\ge 3$, Proposition \ref{non_stable_prop} implies that $T_\sF$ is stable for $d\ge3$.
\end{proof}

%-------------------------------------------------------------------

\subsection{Tangent sheaf of logarithmic foliations}\label{seclog}

We denote by $\mathcal{L}(d_1,\dots,d_r)$ the set of codimension one foliations represented by 1-forms of the type
$$ \omega=f_1\cdots f_r \sum_{i=1}^r \lambda_i\frac{df_i}{f_i} $$
for some sections $f_i\in H^0(\O(d_i))$ ($1\leq d_1\leq \cdots \leq d_r$) and $\lambda_i\in\C^{\ast}$ such that
$\sum d_i\lambda_i=0$; we also assume that the sections $f_i$ have no  multiple factors. It is clear that
$\omega\in H^0(\Omega^1_{\p3}(d_1+\cdots+d_r))$. Observe that the case $r=2$ reduces to the case of rational foliations, considered above; therefore from now on we assume that $r\geq3$. We remark also, cf. \cite{C}:
$$
 \dim \mathcal{L}(d_1,\dots,d_r) =  
\sum_{i=1}^{r}{3+d_i \choose d_i} -2. 
$$

%The singular locus consist of the curve $C=\cup_{i<j}\{f_i=f_j=0\}$, it has $r(r-1)/2$ irreducible components and has degree $\deg(C)=\sum_{i<j} d_id_j$.

Let $T_\sF$ be the tangent sheaf of a logarithmic foliation $\sF\in\mathcal{L}(d_1,\dots,d_r)$; its degree is $d=d_1+\dots+d_r-2$. By \cite[Theorem 3]{CSV}, its Chern classes are
\[
c_1(T_\sF)=2-d,\quad c_2(T_\sF)=d^2+2-\sum_{i<j}d_id_j,\quad c_3(T_\sF)=N(3,d)
\]
where
\[
N(3,d)=\mbox{coefficient of }h^3 \mbox{ in }\frac{(1-h)^4}{(1-d_1h)\cdots(1-d_rh)}
\]

It follows from the equation (\ref{eq:length}) that
\[
N(3,d)=d^3+2d^2+2d-(3d-2)\sum_{i<j}d_id_j+2p_a(C)-2.
\]
One can then compute the arithmetic genus $p_a(C)$ by comparing these last two expressions for $N(3,d)$.

\begin{Theorem} \label{logStable}
Let $\mathscr{F}\in  \mathcal{L}(d_1,\dots,d_r)$ be a logarithmic foliation of degree $d$ on $\P^3$,
\begin{enumerate}
\item[(i)] If $d=1$, then $T_\sF\simeq \O\oplus\O(1)$.
\item[(ii)] If $d=2$, then 
\begin{itemize}
\item either $\sF\in\mathcal{L}(1,1,1,1)$, in which case $T_\sF\simeq \O\oplus\O$;
\item or $\sF\in\mathcal{L}(1,1,2)$, in which case $T_\sF$ is a strictly $\mu$-semistable reflexive sheaf with Chern classes $c_1(T_\sF)=0$,  $c_2(T_\sF)=1$, and $c_3(T_\sF)=2$.
\end{itemize}
\item[(iii)] For $d>2$, $T_\sF$ is stable provided
\begin{itemize}
\item $d>3r-2$ when $d$ is even;
\item $d>2r-2$ when $d$ is odd.
\end{itemize}
\item[(iv)] For $d>2$ and even, $T_\sF$ is $\mu$-semistable provided $d>r-2$.
\end{enumerate}
\end{Theorem}

In particular, we conclude from item (iv) above that the tangent sheaf of a logarithmic foliation $\sF\in \mathcal{L}(d_1,\dots,d_r)$ of even degree is $\mu$-semistable whenever $d_r>1$.

\begin{proof}
When $d=1$, the only possibility is to have $\sF\in\mathcal{L}(1,1,1)$; using the formulas for the Chern classes of $T_\sF$ given above, we obtain $c_1(T_\sF)=1$,  $c_2(T_\sF)=0$, and $c_3(T_\sF)=0$. Comparing with Table \ref{table deg 1}, we conclude that $T_\sF$ splits as a sum of line bundles, and the first item follows immediately.

Suppose that $d=2$. If $\sF\in\mathcal{L}(1,1,1,1)$, then the  Chern classes of $T_\sF$ are
$c_1(T_\sF)=c_2(T_\sF)=c_3(T_\sF)=0$, hence clearly $T_\sF\simeq \O\oplus\O$.

The second possibility is $\sF\in\mathcal{L}(1,1,2)$, in which case we obtain $c_1(T_\sF)=0$, $c_2(T_\sF)=1$, and $c_3(T_\sF)=2$; the fact that $T_\sF$ must be strictly $\mu$-semistable was observed in Remark \ref{rmk 9.4}.

Finally, suppose that $d\geq 2$. It follows from Remark \ref{cota_Log} the degree of the singular curve is bounded by
$$ 2\deg(C) = 2 \sum_{i<j} d_id_j \leq \frac{r(r-1)}{r^2}(d+2)^2=\left(1-\frac{1}{r}\right)(d+2)^2. $$
If $d$ is even and larger than $3r-2$, then the last term in the previous inequality is bounded above by $d^2+d$, hence it follows from Proposition \ref{non_stable_prop} that $T_\sF$ is stable. If $d$ is odd and larger than $2r-2$, then the last term in the previous inequality is bounded above by $(d+1)^2$, and again Proposition \ref{non_stable_prop} allows us to conclude that $T_\sF$ is stable. If $d$ is even and larger than $r-2$, then the last term in the previous inequality is bounded above by $d^2+3d+2$, so $T_\sF$ must be $\mu$-semistable by Proposition \ref{non_stable_prop}.
\end{proof}

We summarize our description of the tangent sheaves of rational and logarithmic foliations of degree $\le2$ in Table \ref{rat log table} below.

\begin{table}[h] 
\begin{tabular}{ ||l || c | c | c | c| c| }\hline
Component               & $c_1(T_\sF)$ & $c_2(T_\sF)$ & $c_3(T_\sF)$ & Tangent sheaf is    \\ \hline
$\mathcal{R}(1,1)$      &  2           &  1           &   0          & $\O(1)\oplus\O(1)$  \\ \hline
$\mathcal{L}(1,1,1)$    &  1           &  1           &   0          & $\O\oplus\O(1)$     \\ \hline
$\mathcal{R}(1,2)$      &  1           &  1           &   1          & stable              \\ \hline
$\mathcal{L}(1,1,1,1)$  &  0           &  0           &   0          & $\O\oplus\O$        \\ \hline
$\mathcal{L}(1,1,2)$    &  0           &  1           &   2          & strictly $\mu$-semistable \\ \hline  
$\mathcal{R}(2,2)$      &  0           &  2           &   4          & stable              \\ \hline  
$\mathcal{R}(1,3)$      &  0           &  3           &   8          & stable              \\ \hline  
%$\mathscr{L}(1,1,1,1,1)$& -1  &  1  &   1 & stable             \\ \hline
%$\mathscr{L}(1,1,1,2)$  & -1  &  2  &   4  & stable             \\ \hline
%$\mathscr{L}(1,2,2)$    & -1  &  3  &   3  & stable         \\ \hline  %Ch
%$\mathscr{L}(1,1,3)$    & -1  &  4  &  12  & stable             \\ \hline
%$\mathscr{R}(2,3)$      & -1  &  5  &  15  & stable             \\ \hline
%$\mathscr{R}(1,4)$      & -1  &  7  &  27   & stable             \\ \hline
\end{tabular}
\caption{Logarithmic and rational components of degree $\le2$.}
\label{rat log table}
\end{table}

%%%%%%%%%%%%%%%%%%%%%%%%%%%%%%%%%%%%%%%%%%%%%%%%%%%%%%%%%%%%%%%%%%%%%%%%%%%%%%%%%%%%%%%%%%%%%
%%%%%%%%%%%%%%%%%%%%%%%%%%%%%%%%%%%%%%%%%%%%%%%%%%%%%%%%%%%%%%%%%%%%%%%%%%%%%%%%%%%%%%%%%%%%%

\section{Descriptions of moduli spaces of distributions}\label{sec:modspc}

In this section, we provide explicit descriptions of certain moduli spaces of codimension one distributions on $\p3$ of degree at most two. We denote the moduli space ${\mathcal D}^{P}$ with  
$$ P := \frac{1}{3}(t+3)(t+2)(t+1)+\frac{1}{2}(t+2)(t+1)(2-d)-(t+2)c+\frac{1}{2}(l +(d-2)c) $$
by ${\mathcal D}(d,c,l)$; note that $d$ is the degree of $\sF$, $c=c_2(T_\sF)$ and $l=c_3(T_\sF)$. Similarly, ${\mathcal D}^{st}(d,c,l):={\mathcal D}^{P,st}$

Our main tool is the forgetful morphism described in Lemma \ref{forget phi}
\begin{equation} \label{forget map} 
\varpi :  {\mathcal D}^{st}(d,c,l) \to {\mathcal M}(2-d,c,l) ~~,~~ \varpi([\sF]) := [T_\sF]
\end{equation}
where ${\mathcal M}(2-d,c,l)$ denotes the moduli space of stable rank 2 reflexive sheaves on $\p3$ with Chern classes
$(c_1,c_2,c_3)=(2-d,c,l)$. We will check that, for certain particular values of the invariants $(d,c,l)$, the forgetful morphism $\varpi$ is surjective, and that ${\mathcal M}(2-d,c,l)$ is an irreducible, %rational, 
nonsingular quasi-projective variety; if, in addition, the dimension of $\Hom(E,T\P^3)$ is constant for every $[E]\in{\mathcal M}(2-d,c,l)$, then one can conclude that ${\mathcal D}^{st}(d,c,l)$ is also irreducible and nonsingular, as well as compute $\dim {\mathcal D}^{st}(d,c,l)$.

%---------------------------------------------------------------------

\subsection{Null-correlation distributions} \label{sec:nc-dist}

A \emph{null-correlation distribution} $\sF$ on $\p3$ is a codimension one distribution whose tangent sheaf is isomorphic to a null-correlation bundle up to twist. To be more precise, assume that $T_\sF\simeq N(-a)$ for some null-correlation bundle $N$; one can check that $\Hom(N(-a),T\p3)=H^0(N\otimes T\p3(a))=0$ for $a\le-2$; on the other hand, according to Example \ref{ex nc bdl}, such distributions do exist for each $a\ge-1$. 

In fact, the moduli space of null-correlation distributions, that is, those of the form 
$$ \sF ~:~ 0\to N(-a) \to T\p3 \to \IZ(4+2a) \to 0, $$
is precisely ${\mathcal D}^{st}(2+2a,1+a^2,0)$. Indeed, note that $c_1(N(-a))=-2a$ and $c_2(N(-a))=1+a^2$, meaning that $[\sF]\in{\mathcal D}^{st}(2+2a,1+a^2,0)$. Conversely, if $[\sF]\in {\mathcal D}^{st}(2+2a,1+a^2,0)$, then $T_\sF(a)$ is a stable rank 2 locally free sheaf with $c_1(T_\sF(a))=0$ and $c_2(T_\sF(a))=1$, hence $T_\sF(a)$ is a null-correlation bundle.

\begin{Prop} \label{space_null}
${\mathcal D}^{st}(2+2a,1+a^2,0)$ is an irreducible, nonsingular quasi-projective variety of dimension 
$$ 8{a+4 \choose 3}-2{a+3\choose 3}-3a-6. $$
\end{Prop}
\begin{proof}
We check that every null-correlation bundle $N$ satisfies 
$$ \Ext^1(N(-a),T\p3) \simeq H^1(N\otimes T\p3(a))=0 ~~{\rm and}~~
\Ext^2(N(-a),N(-a))\simeq H^2(N\otimes N) = 0; $$
Indeed, the second vanishing is a consequence of the smoothness of ${\mathcal M}(-2a,1+a^2,0) \simeq {\mathcal M}(0,1,0)$ (isomorphism is given by twisting by $\O(a)$), which is precisely the complement of a quadric in $\p5$ \cite[Theorem 4.3.4]{OSS}. Twisting the Euler sequence by $N(a)$ and taking cohomology, we obtain 
$$ H^1(N(a+1))^{\oplus 4} \to H^1(N\otimes T\p3(a)) \to H^2(N(a)); $$
since $H^1(N(a+1))=H^2(N(a))=0$ for $a\ge-1$, we conclude that $H^1(N\otimes T\p3(a))=0$ in this range as well. By Lemma \ref{forget phi}, we have therefore proved that ${\mathcal D}^{st}(2+2a,1+a^2,0)$ is nonsingular.

The fibers of the map 
$$ \varpi : {\mathcal D}^{st}(2+2a,1+a^2,0) \to
{\mathcal M}(-2a,1+a^2,0) \simeq {\mathcal M}(0,1,0) ~~,~~ 
\varpi([\sF]) := [T_\sF(a)] $$
are an open subset of $\mathbb{P}H^0(N\otimes T\P^3(a))$. To compute its dimension, we twist the Euler sequence by $N(a)$ and take cohomology, obtaining
$$
0\to H^0(N(a)) \to H^0(N(a+1))^{\oplus 4} \to H^0(N\otimes T\P^3(a)) \to H^1(N(a)) \to 0
$$
for all $a\ge-1$.

If $a=-1$, then $h^0(N\otimes T\P^3(-1))=h^1(N(-1))=1$, and we conclude that
${\mathcal D}^{st}(0,2,0)\simeq{\mathcal M}(0,1,0)$.

If $a\ge0$, then $h^0(N\otimes T\P^3(a))=4h^0(N(a+1))-h^0(N(a))$, and the dimension formula follows from
$h^0(N(k))=2{k+3\choose 3}-(k+2).$
\end{proof}

We observe that ${\mathcal D}^{st}(2+2a,1+a^2,0)={\mathcal D}(2+2a,1+a^2,0)$ for $a=-1,0$. However, this is certainly false when $a$ is sufficiently large: as it was pointed out in Example \ref{exeNonsemi}, there exists a distribution in ${\mathcal D}(14,37,0)$, ie. for $a=6$, whose tangent sheaf is not $\mu$-semistable.

%---------------------------------------------------------------------

\subsection{Distributions whose singular set is a conic and a point} \label{sec:dist deg 1}

Let $\sF$ be a codimension one distribution on $\P^3$ whose singular set consists of a (possibly reducible or non-reduced) conic plus a point. As pointed out in Table \ref{table deg 1}, the tangent sheaf $T_\sF$ is a stable reflexive sheaf with $c_1(T_\sF)=c_2(T_\sF)=c_3(T_\sF)=1$. Thus ${\mathcal D}(1,1,1)={\mathcal D}^{st}(1,1,1)$ and we obtain the forgetful morphism
\begin{equation} \label{111}
\varpi :  {\mathcal D}(1,1,1) \to {\mathcal M}(1,1,1) \simeq {\mathcal M}(-1,1,1),
\end{equation}
where the last isomorphism is given by twisting by $\O(-1)$. 

\begin{Prop}
${\mathcal D}(1,1,1)$ is an irreducible, nonsingular quasi-projective variety of dimension 14.
%$\mathbb{P}^{11}$-bundle over $\mathbb{P}^{3}$.  
\end{Prop} 

\begin{proof}
We check that the conditions of Lemma \ref{forget phi 2} are satisfied. By Theorem \ref{conic+point}, the forgetful morphism $\varpi$ in (\ref{111}) is surjective; in addition, ${\mathcal M}^{\rm st}(-1,1,1)\simeq \p3$, see \cite[proof of Lemma 9.3]{H2}, which also implies that $\Ext^2(E,E)=0$ for every $[E]\in{\mathcal M}(1,1,1)$. To compute $\dim\Hom(E,T\P^3)=h^0(E\otimes T\P^3(-1))$, twist the sequence (\ref{exsqc}) by $T\P^3(-1)$ and pass to cohomology to conclude that $h^0(E\otimes T\P^3(-1))=3\cdot h^0(T\P^3(-1))=12$. It follows that ${\mathcal D}(1,1,1)$ is irreducible and of dimension 14.  
\end{proof}

According to Theorem \ref{racionais}, the tangent sheaf $T_\sF$ of a rational foliation $\sF\in\mathcal{R}(1,2)$ is a stable reflexive sheaf with $c_1(T_\sF)=c_2(T_\sF)=c_3(T_\sF)=1$. It follows that ${\mathcal D}(1,1,1)$ contains $\mathcal{R}(1,2)$ as a closed subset of codimension 3.

%---------------------------------------------------------------------

\subsection{Degree one distributions with isolated singularities}
 
Let $\sF$ be a codimension one distribution on $\P^3$ with only isolated singularities. As it was pointed out in Theorem \ref{generic_deg1}, the tangent sheaf $T_\sF$ is a stable reflexive sheaf with Chern classes $c_1(T_\sF)=1$, $c_2(T_\sF)=3$, and $c_3(T_\sF)=5$. In addition, $T_\sF(1)$ is globally generated, and $h^0(T\P^3\otimes T_\sF^{\vee})=1$, so that the forgetful morphism
$$ \varpi :  {\mathcal D}(1,3,5) \to {\mathcal M}(1,3,5) \simeq {\mathcal M}(-1,3,5) $$
is bijective to an open subset of ${\mathcal M}(-1,3,5)$, cf. \cite[Theorem 3.14]{Chang}; the last isomorphism is given by twisting by $\O(-1)$. 

Note that ${\mathcal D}(1,3,5)$ is a Zariski open subset of $\P H^0(\Omega_{\P^3}^1(3))$; it is therefore clear that $\dim{\mathcal D}(1,3,5)=19$. 

% and $\dim \P H^0(\Omega_{\P^3}^1(3))=19$, this provides a new proof for the unirationality  of ${\mathcal M}(-1,3,5)$, cf. \cite[Theorem 3.14]{Chang}.

%---------------------------------------------------------------------

\subsection{Distributions induced by charge 2 instantons} \label{sec:2-inst}

Let $[\sF]\in{\mathcal D}^{st}(2,2,0)$, so that $T_\sF$ is a stable rank 2 bundle with $c_2=2$; since every rank 2 bundle $E$ on $\p3$ with $(c_1(E),c_2(E))=(0,2)$ is an instanton bundle of charge 2 \cite[Section 9]{H1}, it follows that $T_\sF$ isomorphic to an instanton bundle of charge 2.

Conversely, as it was established in the proof of Theorem \ref{classification_degre_tow}, every rank 2 instanton bundle is the tangent sheaf of a codimension one distribution $[\sF]\in{\mathcal D}^{st}(2,2,0)$. In other words, the forgetful morphism
$$ \varpi :  {\mathcal D}^{st}(2,2,0) \to {\mathcal M}(0,2,0) $$
is surjective.

\begin{Prop}\label{space-inst}
${\mathcal D}^{st}(2,2,0)$ is an irreducible, nonsingular quasi-projective variety of dimension of dimension 22.
\end{Prop}

\begin{proof}
The moduli space of charge 2 instantons coincides with ${\mathcal M}^{\rm st}(0,2,0)$ which is an irreducible, rational, nonsingular quasi-projective variety of dimension 13 \cite[Section 9]{H1}; in particular, we have that 
$\dim\Ext^1(F,F)=13$ and $\Ext^2(F,F)=0$ for every charge 2 instanton bundle $F$. In addition, twisting the Euler sequence by $F$ one concludes that $\Ext^1(F,T\p3)=H^1(F\otimes T\p3)=0$, so one can apply Lemma \ref{forget phi}. To compute $\dim\Hom(F,T\p3)=h^0(F\otimes T\p3)$ we use the exact sequence 
$$ 0 \to H^0(F(1)^{\oplus 4}) \to H^0(F\otimes T\p3) \to H^1(F) \to 0 , $$
(recall that $h^0(F)=h^1(F(1))=0$); it follows that
$$ h^0(F\otimes T\p3) = 4h^0(F) + h^1(F) = 10. $$
\end{proof}

We observe that it is not clear to us whether ${\mathcal D}(2,2,0)={\mathcal D}^{st}(2,2,0)$, that is whether there exists a codimension one distribution $\sF$ of degree 2 whose tangent sheaf is a strictly $\mu$-semistable locally free sheaf with $c_2(T_\sF)=2$. 

%---------------------------------------------------------------------

\subsection{Space of distributions containing $\mathcal{R}(2,2)$} \label{sec:r(2,2)}

According to Theorem \ref{racionais}, the tangent sheaf $T_\sF$ of a rational foliation $\sF\in\mathcal{R}(2,2)$ is a stable reflexive sheaf with $c_1(T_\sF)=c_2(T_\sF)=2$ and $c_3(T_\sF)=4$. We will see in the next result that  ${\mathcal D}^{st}(2,2,4)$ contains $\mathcal{R}(2,2)$ as a closed subset of codimension 12. 

\begin{Prop}
${\mathcal D}^{st}(2,2,4)$ is an irreducible, nonsingular quasi-projective variety of dimension 28. In particular, ${\mathcal D}^{st}(2,2,4)$ contains $\mathcal{R}(2,2)$ as a closed subset of codimension 12.
\end{Prop}
\begin{proof}
Again, we apply Lemma \ref{forget phi 2}. Every stable rank 2 reflexive sheaf $E$ on $\P^3$ with $c_1(E)=0$, $c_2(E)=2$ and $c_3(E)=4$ has a resolution of the form, cf. \cite[Lemma 2.9]{Chang}:
$$ 0\to \O(-1)^{\oplus 2 }  \to \O^{\oplus 4 }  \to E(1)\to 0. $$
In particular, $E(1)$ is globally generated, hence it follows from Proposition \ref{global} that $E$ is the tangent sheaf of a codimension one distribution. In other words, the forgetful morphism
$$ \varpi :  {\mathcal D}^{st}(2,2,4) \to {\mathcal M}(0,2,4) $$
is surjective. Recall also that ${\mathcal M}(0,2,4)$ is an irreducible, rational, nonsingular quasi-projective variety of dimension 13, see \cite[Theorem 2.12]{Chang}, thus $\dim\Ext^1(E,E)=13$ and $\Ext^2(E,E)=0$. 

Twist the resolution of $E(1)$ above by $T\mathbb{P}^3(-1)$ and pass to cohomology to obtain 
$$ H^0(T\mathbb{P}^3(-2))^{\oplus 2}=0  \to H^0(T\mathbb{P}^3(-1))^{\oplus 4}  \to
H^0(E\otimes T\mathbb{P}^3) \to 0=H^1(T\mathbb{P}^3(-2))^{\oplus 2 }. $$
It follows that $\dim\Hom(E,T\p3)=h^0(E\otimes T\mathbb{P}^3)=4\cdot h^0(T\mathbb{P}^3(-1))=16$.
\end{proof}

%---------------------------------------------------------------------

\subsection{Space of distributions containing $\mathcal{R}(1,3)$} \label{sec:r(1,3)}

According to \cite[proof of Theorem 3.9]{Chang}, every stable rank 2 reflexive sheaf $E$ with $c_1(E)=0$, $c_2(E)=3$ and $c_3(E)=8$ has a resolution of the form 
\begin{equation}\label{(3,8) sqc}
0 \to \O(-2) \to \O^{\oplus 3} \to E(1)\to 0.
\end{equation}
In particular, $E(1)$ is globally generated, hence it follows from Proposition \ref{global} that every such sheaf is the tangent sheaf of a codimension one distribution. This means that the morphism 
$$ \varpi :  {\mathcal D}^{\rm st}(2,3,8) \to {\mathcal M}(0,3,8) $$
is surjective. Recall that  ${\mathcal M}(0,3,8)$ is an irreducible, rational, nonsingular quasi-projective variety of dimension 21 \cite[Theorem 3.9]{Chang}, thus in particular $\dim\Ext^1(E,E)=21$ and $\Ext^2(E,E)=0$ for every $[E]\in{\mathcal M}(0,3,8)$.

\begin{Prop}
${\mathcal D}^{\rm st}(2,3,8)$ is an irreducible, nonsingular quasi-projective variety of dimension 32. In particular, ${\mathcal D}^{\rm st}(2,3,8)$ contains the variety of rational foliations $\mathcal{R}(1,3)$ as a closed subset of codimension 11.
\end{Prop}

\begin{proof}
To complete the proof, invoking Lemma \ref{forget phi 2}, twist sequence (\ref{(3,8) sqc}) by $T\mathbb{P}^3(-1)$ and pass to cohomology to obtain
$$ H^0(T\mathbb{P}^3(-3))=0 \to H^0(T\mathbb{P}^3(-1))^{\oplus 3} \to H^0 (E\otimes T\mathbb{P}^3) \to 0=H^1(T\mathbb{P}^3(-3)). $$
It follows that
$$ \dim \Hom(E,T\mathbb{P}^3) = 3\cdot h^0(T\mathbb{P}^3(-1)) = 12. $$
\end{proof}

%%%%%%%%%%%%%%%%%%%%%%%%%%%%%%%%%%%%%%%%%%%%%%%%%%%%%%%%%%%%%%%%%%%%%%%%%%%%%%%%%%%%%%%%%%
%%%%%%%%%%%%%%%%%%%%%%%%%%%%%%%%%%%%%%%%%%%%%%%%%%%%%%%%%%%%%%%%%%%%%%%%%%%%%%%%%%%%%%%%%%

\appendix
\section{A Bertini type Theorem for reflexive sheaves}

Let $X$ be an irreducible projective manifold of dimension $n$; let $G$ be a coherent sheaf on $X$, and let $T$ be a locally free sheaf on $X$. Set $g:=\rk(G)$ and $t:=\rk(T)$; assume that  $g<t$. The \emph{degeneracy locus} of a morphism of sheaves $\phi: G \to T $ is defined by
$$ D(\phi) := \supp(\coker(\phi^{\vee})) . $$

Our next result is a generalization of a \emph{Bertini-type theorem} due to Ottaviani, cf. \cite[Teorema 2.8]{O}.

\begin{appxlem}\label{LemmaGlobal1}
If $G^{\vee}\otimes T$ is globally generated, then, for a generic morphism $\phi\in\Hom(G, T)$, the following holds:
$$ \codim_X(D(\phi)) \leq  min\{ t-g+1, \codim_X \sing(G^{\vee} )  \}.$$
%either $D(\phi) \subset \sing(G^{\vee} )$ or $\codim_X(D(\phi))=\min\{ (g-1)(t-1),\codim_X \sing((G^{\vee} ))  \}$.
\end{appxlem}

\begin{proof}
Let $U:= X \setminus \sing((G^{\vee}))$. Note that $(G^{\vee} \otimes  T)_{|U}$ is locally free and globally generated, since $G^{\vee} \otimes  T$ is globally generated by hypothesis. 
 
Let $V((G^{\vee}\otimes T)_{|U} )$ be the total space of $(G^{\vee}\otimes T)_{|U} $ as a vector bundle over $U$; we have a projection $\pi : V((G^{\vee}\otimes T)_{|U} ) \to U$, whose fibers can be identified with the space of $g\times t$ matrices.  There exists a subvariety $\Sigma \subset V((G^{\vee}\otimes T)_{|U} ) $ equipped with a projection $\pi':\Sigma \to U$ and whose fibres are given by the subspaces of  $g\times t$ matrices of rank is at most $g-1$. Note that $\codim_U (\Sigma )=t-g+1$.

Since $(G^{\vee}\otimes T)_{|U}$ is globally generated, we have the exact sequence:
$$ H^0((G^{\vee}\otimes T)_{|U})\otimes \mathcal{O}_U \to 
(G^{\vee} \otimes T)_{|U} \to 0 $$
Since $G^{\vee} $ is reflexive, $G^{\vee} \otimes  T$ is also reflexive, hence $H^0((G^{\vee}\otimes T)_{|U})\simeq H^0(G^{\vee} \otimes T)$. Therefore, we obtain a submersion 
$$ p: U\times H^0(G^{\vee} \otimes T) \to  V((G^{\vee} \otimes T)_{|U} ) .$$
Let $\Sigma'=p^{-1}(\Sigma) $ and observe that 
$$ \Sigma'=\{(x,\phi); \ \rk(\phi(x))\leq g-1\} . $$
Thus, we have a projection $q: \Sigma' \to H^0(G^{\vee} \otimes T) $, where 
$$ q^{-1}(\phi)=\{ x\in U; \ \rk(\phi(x)) \leq g-1 \} \simeq D(\phi) \cap U. $$
Note that $\codim_{U\times H^0(G^{\vee} \otimes T) } (\Sigma' )= (t-g+1)+h$, where
$h=\dim H^0(G^{\vee} \otimes T) $. There are now two possibilities: 
 
\begin{enumerate}
\item[(i)] if the image of $q$ is not  dense in $H^0(G^{\vee} \otimes T)$, then $D(\phi) \cap U=\emptyset$, or equivalently,  $D(\phi)\subset\sing(G^{\vee})$, for generic $\phi$.
\item[(ii)] if the image of $q$ is dense in $H^0(G^{\vee} \otimes T)$, then a for generic $\phi $ we have that $\Sigma'$ is transversal to $q^{-1}(\phi)$, so that $\codim_U(D(\phi)\cap U)=t-g+1$.
\end{enumerate}
Since $D(\phi)=(D(\phi)\cap U) \cup (D(\phi)\cap \sing(G^{\vee}))$, the result follows. 
\end{proof}

\begin{appxlem} \label{LemmaGlobal2}
Assume that $G$ is reflexive and that $\phi$ is injective. If $\codim_X(D(\phi)) \geq 2$, then $\coker(\phi)$ is torsion free. 
\end{appxlem}
\begin{proof}
Let $K:= \coker(\phi)$. Dualizing the sequence 
$0\to G \to  T \to K \to 0$, we obtain 
\begin{enumerate}
\item[(i)] $\inext^1(K,\mathcal{O}_X)\simeq \coker(\phi^{\vee})$, so that
$\codim_X(\supp(\inext^1(K,\mathcal{O}_X)))\geq 2$ by hypothesis;
\item[(ii)] $\inext^p(K,\mathcal{O}_X)\simeq \inext^{p-1}(G,\mathcal{O}_X)$ for $2\le p\le n-1$, thus $\codim_X(\supp(\inext^p(K,\mathcal{O}_X)))\geq p+1$ since $G$ is reflexive;
\item[(iii)] $\inext^n(K,\O_X)=0$.
\end{enumerate}

The desired conclusion follows from \cite[Proposition 1.1.10]{HL}.
\end{proof}

We apply the previous lemmas to construct distributions on $X$.

\begin{appxthm} \label{TeoGlobal1}
Let $G$ be a globally generated reflexive sheaf on a projective variety $X$ such that
$ \rk(G)\leq \dim X-1\ge2$. If $TX\otimes L$ is globally generated, for some line bundle $L$, then 
$G^{\vee}\otimes L^{\vee}$ is the tangent sheaf of a   distribution on $X$ of codimension $n-\rk(G)$ . 
\end{appxthm}
\begin{proof}
Since $G$ and $TX\otimes L$ are globally generated, then so is 
$$ G \otimes TX\otimes L \simeq  (G^{\vee}\otimes L^{\vee})^{\vee} \otimes TX. $$  
By Lemma \ref{LemmaGlobal1} a generic morphism $\phi: G^{\vee}\otimes L^{\vee}\to TX$ satisfies
$\codim_X(D(\phi))=\min\{n-g+1,3\}\geq 2$, since $G$ is reflexive. Lemma \ref{LemmaGlobal2} implies that $\coker(\phi)$ is torsion free.
\end{proof}

Further specializing to the case $X=\p3$ and $\rk(G)=2$, we obtain the following statement, used throughout the main text.

\begin{appxcor} \label{glob gen p3}
Let $G$ be a globally generated rank 2 reflexive sheaf on $\P^3$.
Then $G^{\vee}(1)$ is the tangent sheaf of a codimension one distribution $\sF$ of degree
$c_1(G)$ with $c_2(T_\sF)=c_2(G)-c_1(G)+1$, and $c_3(T_\sF)=c_3(G)$. 
\end{appxcor}

\begin{proof}
If $G$ is a globally generated rank 2 reflexive sheaf on $\P^3$, then $G\otimes T\mathbb{P}^3(-1)$ is also globally generated, since $T\mathbb{P}^3(-1)$ is globally generated. Now, apply the Theorem \ref{TeoGlobal1} with $L=\O(-1)$ to obtain the desired codimension one distribution:
$$ \sF ~:~ 0 \to G^{\vee}(1) \to T\p3 \to \IZ(d+2) \to 0 , $$
where $d=2-c_1(G^{\vee}(1))=c_1(G)$. In addition, note that $c_2(G^{\vee}(1))=c_2(G)-c_1(G)+1$, and $c_3(G^{\vee}(1))=c_3(G)$
\end{proof}

\bibliographystyle{amsalpha}

\end{document}